\documentclass{amsart}
\usepackage{amsfonts}

\newtheorem{theorem}{Theorem}
\theoremstyle{plain}

\newtheorem{corollary}{Corollary}

\newtheorem{example}{Example}

\newtheorem{proposition}{Proposition}
\newtheorem{remark}{Remark}

\numberwithin{equation}{section}
\input{tcilatex}

\begin{document}
\title[ Some Trace Inequalities for Operators in Hilbert Spaces]{ Some Trace
Inequalities for Operators in Hilbert Spaces}
\author{S. S. Dragomir$^{1,2}$}
\address{$^{1}$Mathematics, School of Engineering \& Science\\
Victoria University, PO Box 14428\\
Melbourne City, MC 8001, Australia.}
\email{sever.dragomir@vu.edu.au}
\urladdr{http://rgmia.org/dragomir}
\address{$^{2}$School of Computational \& Applied Mathematics, University of
the Witwatersrand, Private Bag 3, Johannesburg 2050, South Africa}

\begin{abstract}
Some new trace inequalities for operators in Hilbert spaces are provided.
The superadditivity and monotonicity of some associated functionals are
investigated and applications for power series of such operators are given.
Some trace inequalities for matrices are also derived. Examples for the
operator exponential and other similar functions are presented as well.
\end{abstract}

\maketitle

\section{Introduction}

Let $\left( H,\left\langle \cdot ,\cdot \right\rangle \right) $ be a complex
Hilbert space and $\left\{ e_{i}\right\} _{i\in I}$ an \textit{orthonormal
basis} of $H.$ We say that $A\in \mathcal{B}\left( H\right) $ is a \textit{%
Hilbert-Schmidt operator} if 
\begin{equation}
\sum_{i\in I}\left\Vert Ae_{i}\right\Vert ^{2}<\infty .  \label{e.1.1}
\end{equation}%
It is well know that, if $\left\{ e_{i}\right\} _{i\in I}$ and $\left\{
f_{j}\right\} _{j\in J}$ are orthonormal bases for $H$ and $A\in \mathcal{B}%
\left( H\right) $ then%
\begin{equation}
\sum_{i\in I}\left\Vert Ae_{i}\right\Vert ^{2}=\sum_{j\in I}\left\Vert
Af_{j}\right\Vert ^{2}=\sum_{j\in I}\left\Vert A^{\ast }f_{j}\right\Vert ^{2}
\label{e.1.2}
\end{equation}%
showing that the definition (\ref{e.1.1}) is independent of the orthonormal
basis and $A$ is a Hilbert-Schmidt operator iff $A^{\ast }$ is a
Hilbert-Schmidt operator.

Let $\mathcal{B}_{2}\left( H\right) $ the set of Hilbert-Schmidt operators
in $\mathcal{B}\left( H\right) .$ For $A\in \mathcal{B}_{2}\left( H\right) $
we define%
\begin{equation}
\left\Vert A\right\Vert _{2}:=\left( \sum_{i\in I}\left\Vert
Ae_{i}\right\Vert ^{2}\right) ^{1/2}  \label{e.1.3}
\end{equation}%
for $\left\{ e_{i}\right\} _{i\in I}$ an orthonormal basis of $H.$ This
definition does not depend on the choice of the orthonormal basis.

Using the triangle inequality in $l^{2}\left( I\right) ,$ one checks that $%
\mathcal{B}_{2}\left( H\right) $ is a \textit{vector space} and that $%
\left\Vert \cdot \right\Vert _{2}$ is a norm on $\mathcal{B}_{2}\left(
H\right) ,$ which is usually called in the literature as the \textit{%
Hilbert-Schmidt norm}.

Denote \textit{the modulus} of an operator $A\in \mathcal{B}\left( H\right) $
by $\left\vert A\right\vert :=\left( A^{\ast }A\right) ^{1/2}.$

Because $\left\Vert \left\vert A\right\vert x\right\Vert =\left\Vert
Ax\right\Vert $ for all $x\in H,$ $A$ is Hilbert-Schmidt iff $\left\vert
A\right\vert $ is Hilbert-Schmidt and $\left\Vert A\right\Vert
_{2}=\left\Vert \left\vert A\right\vert \right\Vert _{2}.$ From (\ref{e.1.2}%
) we have that if $A\in \mathcal{B}_{2}\left( H\right) ,$ then $A^{\ast }\in 
\mathcal{B}_{2}\left( H\right) $ and $\left\Vert A\right\Vert
_{2}=\left\Vert A^{\ast }\right\Vert _{2}.$

The following theorem collects some of the most important properties of
Hilbert-Schmidt operators:

\begin{theorem}
\label{t.1.1}We have

(i) $\left( \mathcal{B}_{2}\left( H\right) ,\left\Vert \cdot \right\Vert
_{2}\right) $ is a Hilbert space with inner product 
\begin{equation}
\left\langle A,B\right\rangle _{2}:=\sum_{i\in I}\left\langle
Ae_{i},Be_{i}\right\rangle =\sum_{i\in I}\left\langle B^{\ast
}Ae_{i},e_{i}\right\rangle  \label{e.1.4}
\end{equation}%
and the definition does not depend on the choice of the orthonormal basis $%
\left\{ e_{i}\right\} _{i\in I}$;

(ii) We have the inequalities 
\begin{equation}
\left\Vert A\right\Vert \leq \left\Vert A\right\Vert _{2}  \label{e.1.4.a}
\end{equation}
for any $A\in \mathcal{B}_{2}\left( H\right) $ and 
\begin{equation}
\left\Vert AT\right\Vert _{2},\left\Vert TA\right\Vert _{2}\leq \left\Vert
T\right\Vert \left\Vert A\right\Vert _{2}  \label{e.1.4.b}
\end{equation}%
for any $A\in \mathcal{B}_{2}\left( H\right) $ and $T\in \mathcal{B}\left(
H\right) ;$

(iii) $\mathcal{B}_{2}\left( H\right) $ is an operator ideal in $\mathcal{B}%
\left( H\right) ,$ i.e. 
\begin{equation*}
\mathcal{B}\left( H\right) \mathcal{B}_{2}\left( H\right) \mathcal{B}\left(
H\right) \subseteq \mathcal{B}_{2}\left( H\right) ;
\end{equation*}

(iv) $\mathcal{B}_{fin}\left( H\right) ,$ the space of operators of finite
rank, is a dense subspace of $\mathcal{B}_{2}\left( H\right) ;$

(v) $\mathcal{B}_{2}\left( H\right) \subseteq \mathcal{K}\left( H\right) ,$
where $\mathcal{K}\left( H\right) $ denotes the algebra of compact operators
on $H.$
\end{theorem}

If $\left\{ e_{i}\right\} _{i\in I}$ an orthonormal basis of $H,$ we say
that $A\in \mathcal{B}\left( H\right) $ is \textit{trace class} if 
\begin{equation}
\left\Vert A\right\Vert _{1}:=\sum_{i\in I}\left\langle \left\vert
A\right\vert e_{i},e_{i}\right\rangle <\infty .  \label{e.1.5}
\end{equation}%
The definition of $\left\Vert A\right\Vert _{1}$ does not depend on the
choice of the orthonormal basis $\left\{ e_{i}\right\} _{i\in I}.$ We denote
by $\mathcal{B}_{1}\left( H\right) $ the set of trace class operators in $%
\mathcal{B}\left( H\right) .$

The following proposition holds:

\begin{proposition}
\label{p.1.1}If $A\in \mathcal{B}\left( H\right) ,$ then the following are
equivalent:

(i) $A\in \mathcal{B}_{1}\left( H\right) ;$

(ii) $\left\vert A\right\vert ^{1/2}\in \mathcal{B}_{2}\left( H\right) ;$

(ii) $A$ (or $\left\vert A\right\vert )$ is the product of two elements of $%
\mathcal{B}_{2}\left( H\right) .$
\end{proposition}

The following properties are also well known:

\begin{theorem}
\label{t.1.2}With the above notations:

(i) We have 
\begin{equation}
\left\Vert A\right\Vert _{1}=\left\Vert A^{\ast }\right\Vert _{1}\text{ and }%
\left\Vert A\right\Vert _{2}\leq \left\Vert A\right\Vert _{1}  \label{e.1.6}
\end{equation}%
for any $A\in \mathcal{B}_{1}\left( H\right) ;$

(ii) $\mathcal{B}_{1}\left( H\right) $ is an operator ideal in $\mathcal{B}%
\left( H\right) ,$ i.e. 
\begin{equation*}
\mathcal{B}\left( H\right) \mathcal{B}_{1}\left( H\right) \mathcal{B}\left(
H\right) \subseteq \mathcal{B}_{1}\left( H\right) ;
\end{equation*}

(iii) We have%
\begin{equation*}
\mathcal{B}_{2}\left( H\right) \mathcal{B}_{2}\left( H\right) =\mathcal{B}%
_{1}\left( H\right) ;
\end{equation*}

(iv) We have%
\begin{equation*}
\left\Vert A\right\Vert _{1}=\sup \left\{ \left\langle A,B\right\rangle _{2}%
\text{ }|\text{ }B\in \mathcal{B}_{2}\left( H\right) ,\text{ }\left\Vert
B\right\Vert \leq 1\right\} ;
\end{equation*}

(v) $\left( \mathcal{B}_{1}\left( H\right) ,\left\Vert \cdot \right\Vert
_{1}\right) $ is a Banach space.

(iv) We have the following isometric isomorphisms%
\begin{equation*}
\mathcal{B}_{1}\left( H\right) \cong K\left( H\right) ^{\ast }\text{ and }%
\mathcal{B}_{1}\left( H\right) ^{\ast }\cong \mathcal{B}\left( H\right) ,
\end{equation*}%
where $K\left( H\right) ^{\ast }$ is the dual space of $K\left( H\right) $
and $\mathcal{B}_{1}\left( H\right) ^{\ast }$ is the dual space of $\mathcal{%
B}_{1}\left( H\right) .$
\end{theorem}

We define the \textit{trace} of a trace class operator $A\in \mathcal{B}%
_{1}\left( H\right) $ to be%
\begin{equation}
\limfunc{tr}\left( A\right) :=\sum_{i\in I}\left\langle
Ae_{i},e_{i}\right\rangle  \label{e.1.7}
\end{equation}%
where $\left\{ e_{i}\right\} _{i\in I}$ an orthonormal basis of $H.$ Note
that this coincides with the usual definition of the trace if $H$ is
finite-dimensional. We observe that the series (\ref{e.1.7}) converges
absolutely and it is independent from the choice of basis.

The following result collects some properties of the trace:

\begin{theorem}
\label{t.3.1}We have

(i) If $A\in \mathcal{B}_{1}\left( H\right) $ then $A^{\ast }\in \mathcal{B}%
_{1}\left( H\right) $ and 
\begin{equation}
\limfunc{tr}\left( A^{\ast }\right) =\overline{\limfunc{tr}\left( A\right) };
\label{e.1.8}
\end{equation}

(ii) If $A\in \mathcal{B}_{1}\left( H\right) $ and $T\in \mathcal{B}\left(
H\right) ,$ then $AT,$ $TA\in \mathcal{B}_{1}\left( H\right) $ and%
\begin{equation}
\limfunc{tr}\left( AT\right) =\limfunc{tr}\left( TA\right) \text{ and }%
\left\vert \limfunc{tr}\left( AT\right) \right\vert \leq \left\Vert
A\right\Vert _{1}\left\Vert T\right\Vert ;  \label{e.1.9}
\end{equation}

(iii) $\limfunc{tr}\left( \cdot \right) $ is a bounded linear functional on $%
\mathcal{B}_{1}\left( H\right) $ with $\left\Vert \limfunc{tr}\right\Vert
=1; $

(iv) If $A,$ $B\in \mathcal{B}_{2}\left( H\right) $ then $AB,$ $BA\in 
\mathcal{B}_{1}\left( H\right) $ and $\limfunc{tr}\left( AB\right) =\limfunc{%
tr}\left( BA\right) ;$

(v) $\mathcal{B}_{fin}\left( H\right) $ is a dense subspace of $\mathcal{B}%
_{1}\left( H\right) .$
\end{theorem}

Utilising the trace notation we obviously have that 
\begin{equation*}
\left\langle A,B\right\rangle _{2}=\limfunc{tr}\left( B^{\ast }A\right) =%
\limfunc{tr}\left( AB^{\ast }\right) \text{ and }\left\Vert A\right\Vert
_{2}^{2}=\limfunc{tr}\left( A^{\ast }A\right) =\limfunc{tr}\left( \left\vert
A\right\vert ^{2}\right)
\end{equation*}%
for any $A,$ $B\in \mathcal{B}_{2}\left( H\right) .$

Now, for the finite dimensional case, it is well known that the trace
functional is \textit{submultiplicative}, that is, for \textit{positive
semidefinite matrices} $A$ and $B$ in $M_{n}(\mathbb{C})$,%
\begin{equation*}
0\leq \limfunc{tr}(AB)\leq \limfunc{tr}\left( A\right) \limfunc{tr}\left(
B\right) .
\end{equation*}%
Therefore%
\begin{equation*}
0\leq \limfunc{tr}(A^{k})\leq \left[ \limfunc{tr}\left( A\right) \right]
^{k},
\end{equation*}%
where $k$ is any positive integer.

In 2000, Yang \cite{Y} proved a matrix trace inequality%
\begin{equation}
\limfunc{tr}\left[ (AB)^{k}\right] \leq (\limfunc{tr}A)^{k}(\limfunc{tr}%
B)^{k},  \label{e.1.10}
\end{equation}%
where $A$ and $B$ are positive semidefinite matrices over $\mathbb{C}$ of
the same order $n$ and $k$ is any positive integer. For related works the
reader can refer to \cite{Ch}, \cite{C}, \cite{N} and \cite{Y1}, which are
continuations of the work of Bellman \cite{B}.

If $\left( H,\left\langle \cdot ,\cdot \right\rangle \right) $ is a
separable infinite-dimensional Hilbert space then the inequality (\ref%
{e.1.10}) is also valid for any positive operators $A,$ $B\in \mathcal{B}%
_{1}\left( H\right) .$ This result was obtained by L. Liu in 2007, see \cite%
{L}.

In 2001, Yang et al. \cite{YYT} improved (\ref{e.1.10}) as follows:%
\begin{equation}
\limfunc{tr}\left[ (AB)^{m}\right] \leq \left[ \limfunc{tr}\left(
A^{2m}\right) \limfunc{tr}\left( B^{2m}\right) \right] ^{1/2},
\label{e.1.11}
\end{equation}%
where $A$ and $B$ are positive semidefinite matrices over $\mathbb{C}$ of
the same order and $m$ is any positive integer.

In \cite{SA} the authors have proved many trace inequalities for sums and
products of matrices. For instance, if $A$ and $B$ are positive semidefinite
matrices in $M_{n}\left( \mathbb{C}\right) $ then%
\begin{equation}
\limfunc{tr}\left[ (AB)^{k}\right] \leq \min \left\{ \left\Vert A\right\Vert
^{k}\limfunc{tr}\left( B^{k}\right) ,\left\Vert B\right\Vert ^{k}\limfunc{tr}%
\left( A^{k}\right) \right\}  \label{e.1.12}
\end{equation}%
for any positive integer $k$. Also, if $A,B\in M_{n}\left( \mathbb{C}\right) 
$ then for $r\geq 1$ and $p,q>1$ with $\frac{1}{p}+\frac{1}{q}=1$ we have
the following \textit{Young type inequality} 
\begin{equation}
\limfunc{tr}\left( \left\vert AB^{\ast }\right\vert ^{r}\right) \leq 
\limfunc{tr}\left[ \left( \frac{\left\vert A\right\vert ^{p}}{p}+\frac{%
\left\vert B\right\vert ^{q}}{q}\right) ^{r}\right] .  \label{e.1.13}
\end{equation}%
Ando \cite{A} proved a very strong form of Young's inequality - it was shown
that if $A$ and $B$ are in $M_{n}(\mathbb{C})$, then there is a \textit{%
unitary matrix} $U$ such that%
\begin{equation*}
\left\vert AB^{\ast }\right\vert \leq U\left( \frac{1}{p}\left\vert
A\right\vert ^{p}+\frac{1}{q}\left\vert B\right\vert ^{q}\right) U^{\ast },
\end{equation*}%
where $p,$ $q>1$ with $\frac{1}{p}+\frac{1}{q}=1$, which immediately gives
the trace inequality%
\begin{equation}
\limfunc{tr}\left( \left\vert AB^{\ast }\right\vert \right) \leq \frac{1}{p}%
\limfunc{tr}\left( \left\vert A\right\vert ^{p}\right) +\frac{1}{q}\limfunc{%
tr}\left( \left\vert B\right\vert ^{q}\right) .  \label{e.1.13.a}
\end{equation}%
This inequality can also be obtained from (\ref{e.1.13}) by taking $r=1.$

Another H\"{o}lder type inequality has been proved by Manjegani in \cite{M}
and can be stated as follows:%
\begin{equation}
\limfunc{tr}(AB)\leq \left[ \limfunc{tr}(A^{p})\right] ^{1/p}\left[ \limfunc{%
tr}(B^{q})\right] ^{1/q},  \label{e.1.14}
\end{equation}%
where $p,$ $q>1$ with $\frac{1}{p}+\frac{1}{q}=1$ and $A$ and $B$ are
positive semidefinite matrices.

For the theory of trace functionals and their applications the reader is
referred to \cite{Si}.

For other trace inequalities see \cite{BJL}, \cite{Ch}, \cite{FL}, \cite{Le}%
, \cite{SA0} and \cite{UT}.

In this paper, motivated by the above results, some new inequalities for
Hilbert-Schmidt operators in $\mathcal{B}\left( H\right) $ are provided. The
superadditivity and monotonicity of some associated functionals are
investigated and applications for power series of such operators are given.
Some trace inequalities for matrices are also derived. Examples for the
operator exponential and other similar functions are presented as well.

\section{Some Trace Inequalities Via Hermitian Forms}

Let $P$ a selfadjoint operator with $P\geq 0.$ For $A\in \mathcal{B}%
_{2}\left( H\right) $ and $\left\{ e_{i}\right\} _{i\in I}$ an orthonormal
basis of $H$ we have%
\begin{equation*}
\left\Vert A\right\Vert _{2,P}^{2}:=\limfunc{tr}\left( A^{\ast }PA\right)
=\sum_{i\in I}\left\langle PAe_{i},Ae_{i}\right\rangle \leq \left\Vert
P\right\Vert \sum_{i\in I}\left\Vert Ae_{i}\right\Vert ^{2}=\left\Vert
P\right\Vert \left\Vert A\right\Vert _{2}^{2},
\end{equation*}%
which shows that $\left\langle \cdot ,\cdot \right\rangle _{2,P}$ defined by%
\begin{equation*}
\left\langle A,B\right\rangle _{2,P}:=\limfunc{tr}\left( B^{\ast }PA\right)
=\sum_{i\in I}\left\langle PAe_{i},Be_{i}\right\rangle =\sum_{i\in
I}\left\langle B^{\ast }PAe_{i},e_{i}\right\rangle
\end{equation*}%
is a \textit{nonnegative Hermitian form} on $\mathcal{B}_{2}\left( H\right)
, $ i.e. $\left\langle \cdot ,\cdot \right\rangle _{2,P}$ satisfies the
properties:

\textit{(h)} $\left\langle A,A\right\rangle _{2,P}\geq 0$ for any $A\in 
\mathcal{B}_{2}\left( H\right) ;$

\textit{(hh)} $\left\langle \cdot ,\cdot \right\rangle _{2,P}$ is linear in
the first variable;

\textit{(hhh)} $\left\langle B,A\right\rangle _{2,P}=\overline{\left\langle
A,B\right\rangle }_{2,P}$ for any $A,$ $B\in \mathcal{B}_{2}\left( H\right)
. $

Using the properties of the trace we also have the following representations%
\begin{equation*}
\left\Vert A\right\Vert _{2,P}^{2}:=\limfunc{tr}\left( P\left\vert A^{\ast
}\right\vert ^{2}\right) =\limfunc{tr}\left( \left\vert A^{\ast }\right\vert
^{2}P\right)
\end{equation*}%
and%
\begin{equation*}
\left\langle A,B\right\rangle _{2,P}:=\limfunc{tr}\left( PAB^{\ast }\right) =%
\limfunc{tr}\left( AB^{\ast }P\right) =\limfunc{tr}\left( B^{\ast }PA\right)
\end{equation*}%
for any $A,$ $B\in \mathcal{B}_{2}\left( H\right) .$

We start with the following result:

\begin{theorem}
\label{t.2.1}Let $P$ a selfadjoint operator with $P\geq 0,$ i.e. $%
\left\langle Px,x\right\rangle \geq 0$ for any $x\in H.$

(i) For any $A,$ $B\in \mathcal{B}_{2}\left( H\right) $ we have 
\begin{equation}
\left\vert \limfunc{tr}\left( PAB^{\ast }\right) \right\vert ^{2}\leq 
\limfunc{tr}\left( P\left\vert A^{\ast }\right\vert ^{2}\right) \limfunc{tr}%
\left( P\left\vert B^{\ast }\right\vert ^{2}\right)  \label{e.2.2}
\end{equation}%
and%
\begin{align}
& \left[ \limfunc{tr}\left( P\left\vert A^{\ast }\right\vert ^{2}\right) +2%
\func{Re}\limfunc{tr}\left( PAB^{\ast }\right) +\limfunc{tr}\left(
P\left\vert B^{\ast }\right\vert ^{2}\right) \right] ^{1/2}  \label{e.2.3} \\
& \leq \left[ \limfunc{tr}\left( P\left\vert A^{\ast }\right\vert
^{2}\right) \right] ^{1/2}+\left[ \limfunc{tr}\left( P\left\vert B^{\ast
}\right\vert ^{2}\right) \right] ^{1/2};  \notag
\end{align}

(ii) For any $A,$ $B,$ $C\in \mathcal{B}_{2}\left( H\right) $ we have 
\begin{align}
& \left\vert \limfunc{tr}\left( PAB^{\ast }\right) \limfunc{tr}\left(
P\left\vert C^{\ast }\right\vert ^{2}\right) -\limfunc{tr}\left( PAC^{\ast
}\right) \limfunc{tr}\left( PCB^{\ast }\right) \right\vert ^{2}
\label{e.2.4} \\
& \leq \left[ \limfunc{tr}\left( P\left\vert A^{\ast }\right\vert
^{2}\right) \limfunc{tr}\left( P\left\vert C^{\ast }\right\vert ^{2}\right)
-\left\vert \limfunc{tr}\left( PAC^{\ast }\right) \right\vert ^{2}\right] 
\notag \\
& \times \left[ \limfunc{tr}\left( P\left\vert B^{\ast }\right\vert
^{2}\right) \limfunc{tr}\left( P\left\vert C^{\ast }\right\vert ^{2}\right)
-\left\vert \limfunc{tr}\left( PBC^{\ast }\right) \right\vert ^{2}\right] , 
\notag
\end{align}%
\begin{align}
& \left\vert \limfunc{tr}\left( PAB^{\ast }\right) \right\vert \limfunc{tr}%
\left( P\left\vert C^{\ast }\right\vert ^{2}\right)  \label{e.2.5} \\
& \leq \left\vert \limfunc{tr}\left( PAB^{\ast }\right) \limfunc{tr}\left(
P\left\vert C^{\ast }\right\vert ^{2}\right) -\limfunc{tr}\left( PAC^{\ast
}\right) \limfunc{tr}\left( PCB^{\ast }\right) \right\vert  \notag \\
& +\left\vert \limfunc{tr}\left( PAC^{\ast }\right) \limfunc{tr}\left(
PCB^{\ast }\right) \right\vert  \notag \\
& \leq \left[ \limfunc{tr}\left( P\left\vert A^{\ast }\right\vert
^{2}\right) \right] ^{1/2}\left[ \limfunc{tr}\left( P\left\vert B^{\ast
}\right\vert ^{2}\right) \right] ^{1/2}\limfunc{tr}\left( P\left\vert
C^{\ast }\right\vert ^{2}\right)  \notag
\end{align}%
and%
\begin{align}
& \left\vert \limfunc{tr}\left( PAC^{\ast }\right) \limfunc{tr}\left(
PCB^{\ast }\right) \right\vert  \label{e.2.6} \\
& \leq \frac{1}{2}\left[ \left[ \limfunc{tr}\left( P\left\vert A^{\ast
}\right\vert ^{2}\right) \right] ^{1/2}\left[ \limfunc{tr}\left( P\left\vert
B^{\ast }\right\vert ^{2}\right) \right] ^{1/2}+\left\vert \limfunc{tr}%
\left( PAB^{\ast }\right) \right\vert \right] \limfunc{tr}\left( P\left\vert
C^{\ast }\right\vert ^{2}\right) .  \notag
\end{align}
\end{theorem}

\begin{proof}
(i) Making use of the Schwarz inequality for the nonnegative hermitian form $%
\left\langle \cdot ,\cdot \right\rangle _{2,P}$ we have 
\begin{equation*}
\left\vert \left\langle A,B\right\rangle _{2,P}\right\vert ^{2}\leq
\left\langle A,A\right\rangle _{2,P}\left\langle B,B\right\rangle _{2,P}
\end{equation*}%
for any $A,$ $B\in \mathcal{B}_{2}\left( H\right) $ and the inequality (\ref%
{e.2.2}) is proved.

We observe that $\left\Vert \cdot \right\Vert _{2,P}$ is a seminorm on $%
\mathcal{B}_{2}\left( H\right) $ and by the triangle inequality we have 
\begin{equation*}
\left\Vert A+B\right\Vert _{2,P}\leq \left\Vert A\right\Vert
_{2,P}+\left\Vert B\right\Vert _{2,P}
\end{equation*}%
for any $A,$ $B\in \mathcal{B}_{2}\left( H\right) $ and the inequality (\ref%
{e.2.3}) is proved.

(ii) Let $C\in \mathcal{B}_{2}\left( H\right) ,$ $C\neq 0.$ Define the
mapping $\left[ \cdot ,\cdot \right] _{2,P,C}:\mathcal{B}_{2}\left( H\right)
\times \mathcal{B}_{2}\left( H\right) \rightarrow \mathbb{C}$ by 
\begin{equation*}
\left[ A,B\right] _{2,P,C}:=\left\langle A,B\right\rangle _{2,P}\left\Vert
C\right\Vert _{2,P}^{2}-\left\langle A,C\right\rangle _{2,P}\left\langle
C,B\right\rangle _{2,P}.
\end{equation*}%
Observe that $\left[ \cdot ,\cdot \right] _{2,P,C}$ is a nonnegative
Hermitian form on $\mathcal{B}_{2}\left( H\right) $ and by Schwarz
inequality we have%
\begin{align}
& \left\vert \left\langle A,B\right\rangle _{2,P}\left\Vert C\right\Vert
_{2,P}^{2}-\left\langle A,C\right\rangle _{2,P}\left\langle C,B\right\rangle
_{2,P}\right\vert ^{2}  \label{e.2.7} \\
& \leq \left[ \left\Vert A\right\Vert _{2,P}^{2}\left\Vert C\right\Vert
_{2,P}^{2}-\left\vert \left\langle A,C\right\rangle _{2,P}\right\vert ^{2}%
\right] \left[ \left\Vert B\right\Vert _{2,P}^{2}\left\Vert C\right\Vert
_{2,P}^{2}-\left\vert \left\langle B,C\right\rangle _{2,P}\right\vert ^{2}%
\right]  \notag
\end{align}%
for any $A,$ $B\in \mathcal{B}_{2}\left( H\right) ,$ which proves (\ref%
{e.2.4}).

The case $C=0$ is obvious.

Utilising the elementary inequality for real numbers $m,n,p,q$ 
\begin{equation*}
\left( m^{2}-n^{2}\right) \left( p^{2}-q^{2}\right) \leq \left( mp-nq\right)
^{2},
\end{equation*}%
we can easily see that 
\begin{subequations}
\begin{align}
& \left[ \left\Vert A\right\Vert _{2,P}^{2}\left\Vert C\right\Vert
_{2,P}^{2}-\left\vert \left\langle A,C\right\rangle _{2,P}\right\vert ^{2}%
\right] \left[ \left\Vert B\right\Vert _{2,P}^{2}\left\Vert C\right\Vert
_{2,P}^{2}-\left\vert \left\langle B,C\right\rangle _{2,P}\right\vert ^{2}%
\right]  \label{e.2.8} \\
& \leq \left( \left\Vert A\right\Vert _{2,P}\left\Vert B\right\Vert
_{2,P}\left\Vert C\right\Vert _{2,P}^{2}-\left\vert \left\langle
A,C\right\rangle _{2,P}\right\vert \left\vert \left\langle B,C\right\rangle
_{2,P}\right\vert \right) ^{2}  \notag
\end{align}%
for any $A,$ $B,$ $C\in \mathcal{B}_{2}\left( H\right) .$

Since, by Schwarz's inequality we have 
\end{subequations}
\begin{equation*}
\left\Vert A\right\Vert _{2,P}\left\Vert C\right\Vert _{2,P}\geq \left\vert
\left\langle A,C\right\rangle _{2,P}\right\vert
\end{equation*}%
and%
\begin{equation*}
\left\Vert B\right\Vert _{2,P}\left\Vert C\right\Vert _{2,P}\geq \left\vert
\left\langle B,C\right\rangle _{2,P}\right\vert ,
\end{equation*}%
then by multiplying these inequalities we have%
\begin{equation*}
\left\Vert A\right\Vert _{2,P}\left\Vert B\right\Vert _{2,P}\left\Vert
C\right\Vert _{2,P}^{2}\geq \left\vert \left\langle A,C\right\rangle
_{2,P}\right\vert \left\vert \left\langle B,C\right\rangle _{2,P}\right\vert
\end{equation*}%
for any $A,$ $B,$ $C\in \mathcal{B}_{2}\left( H\right) .$

Utilizing the inequalities (\ref{e.2.7}) and (\ref{e.2.8}) and taking the
square root we get%
\begin{align}
& \left\vert \left\langle A,B\right\rangle _{2,P}\left\Vert C\right\Vert
_{2,P}^{2}-\left\langle A,C\right\rangle _{2,P}\left\langle C,B\right\rangle
_{2,P}\right\vert  \label{e.2.9} \\
& \leq \left\Vert A\right\Vert _{2,P}\left\Vert B\right\Vert
_{2,P}\left\Vert C\right\Vert _{2,P}^{2}-\left\vert \left\langle
A,C\right\rangle _{2,P}\right\vert \left\vert \left\langle B,C\right\rangle
_{2,P}\right\vert  \notag
\end{align}%
for any $A,$ $B,$ $C\in \mathcal{B}_{2}\left( H\right) ,$ which proves the
second inequality in (\ref{e.2.5}).

The first inequality is obvious by the modulus properties.

By the triangle inequality for modulus we also have%
\begin{align}
& \left\vert \left\langle A,C\right\rangle _{2,P}\left\langle
C,B\right\rangle _{2,P}\right\vert -\left\vert \left\langle A,B\right\rangle
_{2,P}\right\vert \left\Vert C\right\Vert _{2,P}^{2}  \label{e.2.10} \\
& \leq \left\vert \left\langle A,B\right\rangle _{2,P}\left\Vert
C\right\Vert _{2,P}^{2}-\left\langle A,C\right\rangle _{2,P}\left\langle
C,B\right\rangle _{2,P}\right\vert  \notag
\end{align}%
for any $A,$ $B,$ $C\in \mathcal{B}_{2}\left( H\right) .$

On making use of (\ref{e.2.9}) and (\ref{e.2.10}) we have%
\begin{align*}
& \left\vert \left\langle A,C\right\rangle _{2,P}\left\langle
C,B\right\rangle _{2,P}\right\vert -\left\vert \left\langle A,B\right\rangle
_{2,P}\right\vert \left\Vert C\right\Vert _{2,P}^{2} \\
& \leq \left\Vert A\right\Vert _{2,P}\left\Vert B\right\Vert
_{2,P}\left\Vert C\right\Vert _{2,P}^{2}-\left\vert \left\langle
A,C\right\rangle _{2,P}\right\vert \left\vert \left\langle B,C\right\rangle
_{2,P}\right\vert ,
\end{align*}%
which is equivalent to the desired inequality (\ref{e.2.6}).
\end{proof}

\begin{remark}
\label{r.2.1}By the triangle inequality for the hermitian form $\left[ \cdot
,\cdot \right] _{2,P,C}:\mathcal{B}_{2}\left( H\right) \times \mathcal{B}%
_{2}\left( H\right) \rightarrow \mathbb{C}$, 
\begin{equation*}
\left[ A,B\right] _{2,P,C}:=\left\langle A,B\right\rangle _{2,P}\left\Vert
C\right\Vert _{2,P}^{2}-\left\langle A,C\right\rangle _{2,P}\left\langle
C,B\right\rangle _{2,P}
\end{equation*}%
we have%
\begin{align*}
& \left( \left\Vert A+B\right\Vert _{2,P}^{2}\left\Vert C\right\Vert
_{2,P}^{2}-\left\vert \left\langle A+B,C\right\rangle _{2,P}\right\vert
^{2}\right) ^{1/2} \\
& \leq \left( \left\Vert A\right\Vert _{2,P}^{2}\left\Vert C\right\Vert
_{2,P}^{2}-\left\vert \left\langle A,C\right\rangle _{2,P}\right\vert
^{2}\right) ^{1/2}+\left( \left\Vert B\right\Vert _{2,P}^{2}\left\Vert
C\right\Vert _{2,P}^{2}-\left\vert \left\langle B,C\right\rangle
_{2,P}\right\vert ^{2}\right) ^{1/2},
\end{align*}%
which can be written as%
\begin{align}
& \left( \limfunc{tr}\left[ P\left\vert \left( A+B\right) ^{\ast
}\right\vert ^{2}\right] \limfunc{tr}\left( P\left\vert C^{\ast }\right\vert
^{2}\right) -\left\vert \limfunc{tr}\left[ P\left( A+B\right) C^{\ast }%
\right] \right\vert ^{2}\right) ^{1/2}  \label{e.2.11} \\
& \leq \left( \limfunc{tr}\left( P\left\vert A^{\ast }\right\vert
^{2}\right) \limfunc{tr}\left( P\left\vert C^{\ast }\right\vert ^{2}\right)
-\left\vert \limfunc{tr}\left( PAC^{\ast }\right) \right\vert ^{2}\right)
^{1/2}  \notag \\
& +\left( \limfunc{tr}\left( P\left\vert B^{\ast }\right\vert ^{2}\right) 
\limfunc{tr}\left( P\left\vert C^{\ast }\right\vert ^{2}\right) -\left\vert 
\limfunc{tr}\left( PBC^{\ast }\right) \right\vert ^{2}\right) ^{1/2}  \notag
\end{align}%
for any $A,$ $B,$ $C\in \mathcal{B}_{2}\left( H\right) .$
\end{remark}

\begin{remark}
\label{r.2.2}If we take $B=\lambda C$ in (\ref{e.2.11}), then we get%
\begin{align}
0& \leq \limfunc{tr}\left[ P\left\vert \left( A+\lambda C\right) ^{\ast
}\right\vert ^{2}\right] \limfunc{tr}\left( P\left\vert C^{\ast }\right\vert
^{2}\right) -\left\vert \limfunc{tr}\left[ P\left( A+\lambda C\right)
C^{\ast }\right] \right\vert ^{2}  \label{e.2.12} \\
& \leq \limfunc{tr}\left( P\left\vert A^{\ast }\right\vert ^{2}\right) 
\limfunc{tr}\left( P\left\vert C^{\ast }\right\vert ^{2}\right) -\left\vert 
\limfunc{tr}\left( C^{\ast }PA\right) \right\vert ^{2}  \notag
\end{align}%
for any $\lambda \in \mathbb{C}$ and $A,$ $C\in \mathcal{B}_{2}\left(
H\right) .$

Therefore, we have the bound%
\begin{align}
& \sup_{\lambda \in \mathbb{C}}\left\{ \limfunc{tr}\left[ P\left\vert \left(
A+\lambda C\right) ^{\ast }\right\vert ^{2}\right] \limfunc{tr}\left(
P\left\vert C^{\ast }\right\vert ^{2}\right) -\left\vert \limfunc{tr}\left[
P\left( A+\lambda C\right) C^{\ast }\right] \right\vert ^{2}\right\}
\label{e.2.13} \\
& =\limfunc{tr}\left( P\left\vert A^{\ast }\right\vert ^{2}\right) \limfunc{%
tr}\left( P\left\vert C^{\ast }\right\vert ^{2}\right) -\left\vert \limfunc{%
tr}\left( PAC^{\ast }\right) \right\vert ^{2}.  \notag
\end{align}%
We also have the inequalities%
\begin{align}
0& \leq \limfunc{tr}\left[ P\left\vert \left( A\pm C\right) ^{\ast
}\right\vert ^{2}\right] \limfunc{tr}\left( P\left\vert C^{\ast }\right\vert
^{2}\right) -\left\vert \limfunc{tr}\left[ P\left( A\pm C\right) C^{\ast }%
\right] \right\vert ^{2}  \label{e.2.14} \\
& \leq \limfunc{tr}\left( P\left\vert A^{\ast }\right\vert ^{2}\right) 
\limfunc{tr}\left( P\left\vert C^{\ast }\right\vert ^{2}\right) -\left\vert 
\limfunc{tr}\left( PAC^{\ast }\right) \right\vert ^{2}  \notag
\end{align}%
for any $A,$ $C\in \mathcal{B}_{2}\left( H\right) .$
\end{remark}

\begin{remark}
\label{r.2.3}We observe that, by replacing $A^{\ast }$ with $A,$ $B^{\ast }$
with $B$ etc...above, we can get the dual inequalities, like, for instance%
\begin{align}
& \left\vert \limfunc{tr}\left( PA^{\ast }C\right) \limfunc{tr}\left(
PC^{\ast }B\right) \right\vert  \label{e.2.15} \\
& \leq \frac{1}{2}\left[ \left[ \limfunc{tr}\left( P\left\vert A\right\vert
^{2}\right) \right] ^{1/2}\left[ \limfunc{tr}\left( P\left\vert B\right\vert
^{2}\right) \right] ^{1/2}+\left\vert \limfunc{tr}\left( PA^{\ast }B\right)
\right\vert \right] \limfunc{tr}\left( P\left\vert C\right\vert ^{2}\right) ,
\notag
\end{align}%
that holds for any $A,$ $B,$ $C\in \mathcal{B}_{2}\left( H\right) .$

Since 
\begin{equation*}
\left\vert \limfunc{tr}\left( PA^{\ast }C\right) \right\vert =\left\vert 
\overline{\limfunc{tr}\left( PA^{\ast }C\right) }\right\vert =\left\vert 
\limfunc{tr}\left[ \left( PA^{\ast }C\right) ^{\ast }\right] \right\vert
=\left\vert \limfunc{tr}\left( C^{\ast }AP\right) \right\vert =\left\vert 
\limfunc{tr}\left( PC^{\ast }A\right) \right\vert ,
\end{equation*}%
\begin{equation*}
\left\vert \limfunc{tr}\left( PC^{\ast }B\right) \right\vert =\left\vert 
\limfunc{tr}\left( PB^{\ast }C\right) \right\vert
\end{equation*}%
and%
\begin{equation*}
\left\vert \limfunc{tr}\left( PA^{\ast }B\right) \right\vert =\left\vert 
\limfunc{tr}\left( PB^{\ast }A\right) \right\vert
\end{equation*}%
then the inequality (\ref{e.2.15}) can be also written as%
\begin{align}
& \left\vert \limfunc{tr}\left( PC^{\ast }A\right) \limfunc{tr}\left(
PB^{\ast }C\right) \right\vert  \label{e.2.16} \\
& \leq \frac{1}{2}\left[ \left[ \limfunc{tr}\left( P\left\vert A\right\vert
^{2}\right) \right] ^{1/2}\left[ \limfunc{tr}\left( P\left\vert B\right\vert
^{2}\right) \right] ^{1/2}+\left\vert \limfunc{tr}\left( PB^{\ast }A\right)
\right\vert \right] \limfunc{tr}\left( P\left\vert C\right\vert ^{2}\right) ,
\notag
\end{align}%
that holds for any $A,B,C\in \mathcal{B}_{2}\left( H\right) .$

If we take in (\ref{e.2.16}) $B=A^{\ast }$ then we get the following
inequality 
\begin{align}
& \left\vert \limfunc{tr}\left( PC^{\ast }A\right) \limfunc{tr}\left(
PAC\right) \right\vert  \label{e.2.17} \\
& \leq \frac{1}{2}\left[ \left[ \limfunc{tr}\left( P\left\vert A\right\vert
^{2}\right) \right] ^{1/2}\left[ \limfunc{tr}\left( P\left\vert A^{\ast
}\right\vert ^{2}\right) \right] ^{1/2}+\left\vert \limfunc{tr}\left(
PA^{2}\right) \right\vert \right] \limfunc{tr}\left( P\left\vert
C\right\vert ^{2}\right) ,  \notag
\end{align}%
for any $A,B,C\in \mathcal{B}_{2}\left( H\right) .$

If $A$ is a normal operator, i.e. $\left\vert A\right\vert ^{2}=\left\vert
A^{\ast }\right\vert ^{2}$ then we have from (\ref{e.2.17}) that%
\begin{equation}
\left\vert \limfunc{tr}\left( PC^{\ast }A\right) \limfunc{tr}\left(
PAC\right) \right\vert \leq \frac{1}{2}\left[ \limfunc{tr}\left( P\left\vert
A\right\vert ^{2}\right) +\left\vert \limfunc{tr}\left( PA^{2}\right)
\right\vert \right] \limfunc{tr}\left( P\left\vert C\right\vert ^{2}\right) ,
\label{e.2.17.a}
\end{equation}

In particular, if $C$ is selfadjoint and $C\in \mathcal{B}_{2}\left(
H\right) ,$ then%
\begin{equation}
\left\vert \limfunc{tr}\left( PAC\right) \right\vert ^{2}\leq \frac{1}{2}%
\left[ \limfunc{tr}\left( P\left\vert A\right\vert ^{2}\right) +\left\vert 
\limfunc{tr}\left( PA^{2}\right) \right\vert \right] \limfunc{tr}\left(
PC^{2}\right) ,  \label{e.2.18}
\end{equation}%
for any $A\in \mathcal{B}_{2}\left( H\right) $ a normal operator.

We notice that (\ref{e.2.18}) is a trace operator version of \textit{de
Bruijn inequality} obtained in 1960 in \cite{BR}, which gives the following
refinement of the Cauchy-Bunyakovsky-Schwarz inequality:%
\begin{equation}
\left\vert \sum_{i=1}^{n}a_{i}z_{i}\right\vert ^{2}\leq \frac{1}{2}%
\sum_{i=1}^{n}a_{i}^{2}\left[ \sum_{i=1}^{n}\left\vert z_{i}\right\vert
^{2}+\left\vert \sum_{i=1}^{n}z_{i}^{2}\right\vert \right] ,  \label{dB}
\end{equation}%
provided that $a_{i}$ are real numbers while $z_{i}$ are complex for each $%
i\in \left\{ 1,...,n\right\} .$
\end{remark}

We notice that, if $P\in \mathcal{B}_{1}\left( H\right) ,$ $P\geq 0$ and $A,$
$B\in \mathcal{B}\left( H\right) ,$ then%
\begin{equation*}
\left\langle A,B\right\rangle _{2,P}:=\limfunc{tr}\left( PAB^{\ast }\right) =%
\limfunc{tr}\left( AB^{\ast }P\right) =\limfunc{tr}\left( B^{\ast }PA\right)
\end{equation*}%
is a \textit{nonnegative Hermitian form} on $\mathcal{B}\left( H\right) $
and all the inequalities above will hold for $A,$ $B,$ $C\in \mathcal{B}%
\left( H\right) .$ The details are left to the reader.

\section{Some Functional Properties}

We consider now the convex cone $\mathcal{B}_{+}\left( H\right) $ of
nonnegative operators on the complex Hilbert space $H$ and, for $A,$ $B\in 
\mathcal{B}_{2}\left( H\right) $ define the functional $\sigma _{A,B}:%
\mathcal{B}_{+}\left( H\right) \rightarrow \lbrack 0,\infty )$ by%
\begin{equation}
\sigma _{A,B}\left( P\right) :=\left[ \limfunc{tr}\left( P\left\vert
A\right\vert ^{2}\right) \right] ^{1/2}\left[ \limfunc{tr}\left( P\left\vert
B\right\vert ^{2}\right) \right] ^{1/2}-\left\vert \limfunc{tr}\left(
PA^{\ast }B\right) \right\vert \left( \geq 0\right) .  \label{e.2.21}
\end{equation}%
The following theorem collects some fundamental properties of this
functional.

\begin{theorem}
\label{t.2.2}Let $A,B\in \mathcal{B}_{2}\left( H\right) .$

(i) For any $P,$ $Q\in \mathcal{B}_{+}\left( H\right) $ we have%
\begin{equation}
\sigma _{A,B}\left( P+Q\right) \geq \sigma _{A,B}\left( P\right) +\sigma
_{A,B}\left( Q\right) \left( \geq 0\right) ,  \label{e.2.22}
\end{equation}%
namely, $\sigma _{A,B}$ is a superadditive functional on $\mathcal{B}%
_{+}\left( H\right) ;$

(ii) For any $P,$ $Q\in \mathcal{B}_{+}\left( H\right) $ with $P\geq Q$ we
have%
\begin{equation}
\sigma _{A,B}\left( P\right) \geq \sigma _{A,B}\left( Q\right) \left( \geq
0\right) ,  \label{e.2.23}
\end{equation}%
namely, $\sigma _{A,B}$ is a monotonic nondecreasing functional on $\mathcal{%
B}_{+}\left( H\right) ;$

(iii) If $P,$ $Q\in \mathcal{B}_{+}\left( H\right) $ and there exist the
constants $M>m>0$ such that $MQ\geq $ $P\geq mQ$ then%
\begin{equation}
M\sigma _{A,B}\left( Q\right) \geq \sigma _{A,B}\left( P\right) \geq m\sigma
_{A,B}\left( Q\right) \left( \geq 0\right) .  \label{e.2.24}
\end{equation}
\end{theorem}

\begin{proof}
(i) Let $P,$ $Q\in \mathcal{B}_{+}\left( H\right) $. On utilizing the
elementary inequality%
\begin{equation*}
\left( a^{2}+b^{2}\right) ^{1/2}\left( c^{2}+d^{2}\right) ^{1/2}\geq ac+bd,%
\text{ }a,b,c,d\geq 0
\end{equation*}%
and the triangle inequality for the modulus, we have%
\begin{align*}
& \sigma _{A,B}\left( P+Q\right) \\
& =\left[ \limfunc{tr}\left( \left( P+Q\right) \left\vert A\right\vert
^{2}\right) \right] ^{1/2}\left[ \limfunc{tr}\left( \left( P+Q\right)
\left\vert B\right\vert ^{2}\right) \right] ^{1/2}-\left\vert \limfunc{tr}%
\left( \left( P+Q\right) A^{\ast }B\right) \right\vert \\
& =\left[ \limfunc{tr}\left( P\left\vert A\right\vert ^{2}+Q\left\vert
A\right\vert ^{2}\right) \right] ^{1/2}\left[ \limfunc{tr}\left( P\left\vert
B\right\vert ^{2}+Q\left\vert B\right\vert ^{2}\right) \right] ^{1/2} \\
& -\left\vert \limfunc{tr}\left( PA^{\ast }B+QA^{\ast }B\right) \right\vert
\\
& =\left[ \limfunc{tr}\left( P\left\vert A\right\vert ^{2}\right) +\limfunc{%
tr}\left( Q\left\vert A\right\vert ^{2}\right) \right] ^{1/2}\left[ \limfunc{%
tr}\left( P\left\vert B\right\vert ^{2}\right) +\limfunc{tr}\left(
Q\left\vert B\right\vert ^{2}\right) \right] ^{1/2} \\
& -\left\vert \limfunc{tr}\left( PA^{\ast }B\right) +\limfunc{tr}\left(
QA^{\ast }B\right) \right\vert \\
& \geq \left[ \limfunc{tr}\left( P\left\vert A\right\vert ^{2}\right) \right]
^{1/2}\left[ \limfunc{tr}\left( P\left\vert B\right\vert ^{2}\right) \right]
^{1/2}+\left[ \limfunc{tr}\left( Q\left\vert A\right\vert ^{2}\right) \right]
^{1/2}\left[ \limfunc{tr}\left( Q\left\vert B\right\vert ^{2}\right) \right]
^{1/2} \\
& -\left\vert \limfunc{tr}\left( PA^{\ast }B\right) \right\vert -\left\vert 
\limfunc{tr}\left( QA^{\ast }B\right) \right\vert \\
& =\sigma _{A,B}\left( P\right) +\sigma _{A,B}\left( Q\right)
\end{align*}%
and the inequality (\ref{e.2.22}) is proved.

(ii) Let $P,$ $Q\in \mathcal{B}_{+}\left( H\right) $ with $P\geq Q.$
Utilising the superadditivity property we have%
\begin{eqnarray*}
\sigma _{A,B}\left( P\right) &=&\sigma _{A,B}\left( \left( P-Q\right)
+Q\right) \geq \sigma _{A,B}\left( P-Q\right) +\sigma _{A,B}\left( Q\right)
\\
&\geq &\sigma _{A,B}\left( Q\right)
\end{eqnarray*}%
and the inequality (\ref{e.2.23}) is obtained.

(iii) From the monotonicity property we have 
\begin{equation*}
\sigma _{A,B}\left( P\right) \geq \sigma _{A,B}\left( mQ\right) =m\sigma
_{A,B}\left( Q\right)
\end{equation*}%
and a similar inequality for $M,$ which prove the desired result (\ref%
{e.2.24}).
\end{proof}

\begin{corollary}
\label{c.2.2}Let $A,$ $B\in \mathcal{B}_{2}\left( H\right) $ and $P\in 
\mathcal{B}\left( H\right) $ such that there exist the constants $M>m>0$
with $M1_{H}\geq $ $P\geq m1_{H}.$ Then we have%
\begin{align}
& M\left( \left[ \limfunc{tr}\left( \left\vert A\right\vert ^{2}\right) %
\right] ^{1/2}\left[ \limfunc{tr}\left( \left\vert B\right\vert ^{2}\right) %
\right] ^{1/2}-\left\vert \limfunc{tr}\left( A^{\ast }B\right) \right\vert
\right)  \label{e.2.25} \\
& \geq \left[ \limfunc{tr}\left( P\left\vert A\right\vert ^{2}\right) \right]
^{1/2}\left[ \limfunc{tr}\left( P\left\vert B\right\vert ^{2}\right) \right]
^{1/2}-\left\vert \limfunc{tr}\left( PA^{\ast }B\right) \right\vert  \notag
\\
& \geq m\left( \left[ \limfunc{tr}\left( \left\vert A\right\vert ^{2}\right) %
\right] ^{1/2}\left[ \limfunc{tr}\left( \left\vert B\right\vert ^{2}\right) %
\right] ^{1/2}-\left\vert \limfunc{tr}\left( A^{\ast }B\right) \right\vert
\right) .  \notag
\end{align}
\end{corollary}

Let $P=\left\vert V\right\vert ^{2}$ with $V\in \mathcal{B}\left( H\right) .$
If $A,$ $B\in \mathcal{B}_{2}\left( H\right) $ then%
\begin{align*}
\sigma _{A,B}\left( \left\vert V\right\vert ^{2}\right) & =\left[ \limfunc{tr%
}\left( \left\vert V\right\vert ^{2}\left\vert A\right\vert ^{2}\right) %
\right] ^{1/2}\left[ \limfunc{tr}\left( \left\vert V\right\vert
^{2}\left\vert B\right\vert ^{2}\right) \right] ^{1/2}-\left\vert \limfunc{tr%
}\left( \left\vert V\right\vert ^{2}A^{\ast }B\right) \right\vert \\
& =\left[ \limfunc{tr}\left( V^{\ast }VA^{\ast }A\right) \right] ^{1/2}\left[
\limfunc{tr}\left( V^{\ast }VB^{\ast }B\right) \right] ^{1/2}-\left\vert 
\limfunc{tr}\left( V^{\ast }VA^{\ast }B\right) \right\vert \\
& =\left[ \limfunc{tr}\left( VA^{\ast }AV^{\ast }\right) \right] ^{1/2}\left[
\limfunc{tr}\left( VB^{\ast }BV^{\ast }\right) \right] ^{1/2}-\left\vert 
\limfunc{tr}\left( VA^{\ast }BV^{\ast }\right) \right\vert \\
& =\left[ \limfunc{tr}\left( \left( AV^{\ast }\right) ^{\ast }AV^{\ast
}\right) \right] ^{1/2}\left[ \limfunc{tr}\left( \left( BV^{\ast }\right)
^{\ast }BV^{\ast }\right) \right] ^{1/2}-\left\vert \limfunc{tr}\left(
\left( AV^{\ast }\right) ^{\ast }BV^{\ast }\right) \right\vert \\
& =\left[ \limfunc{tr}\left( \left\vert AV^{\ast }\right\vert ^{2}\right) %
\right] ^{1/2}\left[ \limfunc{tr}\left( \left\vert BV^{\ast }\right\vert
^{2}\right) \right] ^{1/2}-\left\vert \limfunc{tr}\left( \left( AV^{\ast
}\right) ^{\ast }BV^{\ast }\right) \right\vert .
\end{align*}

On utilizing the property (\ref{e.2.22}) for $P=\left\vert V\right\vert
^{2}, $ $Q=\left\vert U\right\vert ^{2}$ with $V,$ $U\in \mathcal{B}\left(
H\right) ,$ then we have for any $A,$ $B\in \mathcal{B}_{2}\left( H\right) $
the following trace inequality 
\begin{align}
& \left[ \limfunc{tr}\left( \left\vert AV^{\ast }\right\vert ^{2}+\left\vert
AU^{\ast }\right\vert ^{2}\right) \right] ^{1/2}\left[ \limfunc{tr}\left(
\left\vert BV^{\ast }\right\vert ^{2}+\left\vert BU^{\ast }\right\vert
^{2}\right) \right] ^{1/2}  \label{e.2.26} \\
& -\left\vert \limfunc{tr}\left( \left( AV^{\ast }\right) ^{\ast }BV^{\ast
}+\left( AU^{\ast }\right) ^{\ast }BU^{\ast }\right) \right\vert  \notag \\
& \geq \left[ \limfunc{tr}\left( \left\vert AV^{\ast }\right\vert
^{2}\right) \right] ^{1/2}\left[ \limfunc{tr}\left( \left\vert BV^{\ast
}\right\vert ^{2}\right) \right] ^{1/2}-\left\vert \limfunc{tr}\left( \left(
AV^{\ast }\right) ^{\ast }BV^{\ast }\right) \right\vert  \notag \\
& +\left[ \limfunc{tr}\left( \left\vert AU^{\ast }\right\vert ^{2}\right) %
\right] ^{1/2}\left[ \limfunc{tr}\left( \left\vert BU^{\ast }\right\vert
^{2}\right) \right] ^{1/2}-\left\vert \limfunc{tr}\left( \left( AU^{\ast
}\right) ^{\ast }BU^{\ast }\right) \right\vert \left( \geq 0\right) .  \notag
\end{align}

Also, if $\left\vert V\right\vert ^{2}\geq \left\vert U\right\vert ^{2}$
with $V,$ $U\in \mathcal{B}\left( H\right) ,$ then we have for any $A,$ $%
B\in \mathcal{B}_{2}\left( H\right) $ that%
\begin{align}
& \left[ \limfunc{tr}\left( \left\vert AV^{\ast }\right\vert ^{2}\right) %
\right] ^{1/2}\left[ \limfunc{tr}\left( \left\vert BV^{\ast }\right\vert
^{2}\right) \right] ^{1/2}-\left\vert \limfunc{tr}\left( \left( AV^{\ast
}\right) ^{\ast }BV^{\ast }\right) \right\vert  \label{e.2.27} \\
& \geq \left[ \limfunc{tr}\left( \left\vert AU^{\ast }\right\vert
^{2}\right) \right] ^{1/2}\left[ \limfunc{tr}\left( \left\vert BU^{\ast
}\right\vert ^{2}\right) \right] ^{1/2}-\left\vert \limfunc{tr}\left( \left(
AU^{\ast }\right) ^{\ast }BU^{\ast }\right) \right\vert \left( \geq 0\right)
.  \notag
\end{align}

If $U\in \mathcal{B}\left( H\right) $ is invertible, then 
\begin{equation*}
\frac{1}{\left\Vert U^{-1}\right\Vert }\left\Vert x\right\Vert \leq
\left\Vert Ux\right\Vert \leq \left\Vert U\right\Vert \left\Vert
x\right\Vert \text{ for any }x\in H,
\end{equation*}%
which implies that%
\begin{equation*}
\frac{1}{\left\Vert U^{-1}\right\Vert ^{2}}1_{H}\leq \left\vert U\right\vert
^{2}\leq \left\Vert U\right\Vert ^{2}1_{H}.
\end{equation*}%
By making use of (\ref{e.2.25}) we have the following trace inequality%
\begin{align}
& \left\Vert U\right\Vert ^{2}\left( \left[ \limfunc{tr}\left( \left\vert
A\right\vert ^{2}\right) \right] ^{1/2}\left[ \limfunc{tr}\left( \left\vert
B\right\vert ^{2}\right) \right] ^{1/2}-\left\vert \limfunc{tr}\left(
A^{\ast }B\right) \right\vert \right)  \label{e.2.28} \\
& \geq \left[ \limfunc{tr}\left( \left\vert AU^{\ast }\right\vert
^{2}\right) \right] ^{1/2}\left[ \limfunc{tr}\left( \left\vert BU^{\ast
}\right\vert ^{2}\right) \right] ^{1/2}-\left\vert \limfunc{tr}\left( \left(
AU^{\ast }\right) ^{\ast }BU^{\ast }\right) \right\vert  \notag \\
& \geq \frac{1}{\left\Vert U^{-1}\right\Vert ^{2}}\left( \left[ \limfunc{tr}%
\left( \left\vert A\right\vert ^{2}\right) \right] ^{1/2}\left[ \limfunc{tr}%
\left( \left\vert B\right\vert ^{2}\right) \right] ^{1/2}-\left\vert 
\limfunc{tr}\left( A^{\ast }B\right) \right\vert \right)  \notag
\end{align}%
for any $A,$ $B\in \mathcal{B}_{2}\left( H\right) .$

Similar results may be stated for $P\in \mathcal{B}_{1}\left( H\right) ,$ $%
P\geq 0$ and $A,$ $B\in \mathcal{B}\left( H\right) .$ The details are
omitted.

\section{Inequalities for Sequences of Operators}

For $n\geq 2,$ define the Cartesian products $\mathcal{B}^{\left( n\right)
}\left( H\right) :=$ $\mathcal{B}\left( H\right) \times ...\times \mathcal{B}%
\left( H\right) ,$ $\mathcal{B}_{2}^{\left( n\right) }\left( H\right) :=%
\mathcal{B}_{2}\left( H\right) \times ...\times \mathcal{B}_{2}\left(
H\right) $ and $\mathcal{B}_{+}^{\left( n\right) }\left( H\right) :=$ $%
\mathcal{B}_{+}\left( H\right) \times ...\times \mathcal{B}_{+}\left(
H\right) $ where $\mathcal{B}_{+}\left( H\right) $ denotes the convex cone
of nonnegative selfadjoint operators on $H,$ i.e. $P\in \mathcal{B}%
_{+}\left( H\right) $ if $\left\langle Px,x\right\rangle \geq 0$ for any $%
x\in H.$

\begin{proposition}
\label{p.3.1}Let $\mathbf{P}=\left( P_{1},...,P_{n}\right) \in \mathcal{B}%
_{+}^{\left( n\right) }\left( H\right) $ and $\mathbf{A}=\left(
A_{1},...,A_{n}\right) ,$ $\mathbf{B}=\left( B_{1},...,B_{n}\right) \in 
\mathcal{B}_{2}^{\left( n\right) }\left( H\right) $ and $\mathbf{z}=\left(
z_{1},...,z_{n}\right) \in \mathbb{C}^{n}$ with $n\geq 2.$ Then%
\begin{equation}
\left\vert \limfunc{tr}\left( \sum_{k=1}^{n}z_{k}P_{k}A_{k}^{\ast
}B_{k}\right) \right\vert ^{2}\leq \limfunc{tr}\left(
\sum_{k=1}^{n}\left\vert z_{k}\right\vert P_{k}\left\vert A_{k}\right\vert
^{2}\right) \limfunc{tr}\left( \sum_{k=1}^{n}\left\vert z_{k}\right\vert
P_{k}\left\vert B_{k}\right\vert ^{2}\right) .  \label{e.3.1}
\end{equation}
\end{proposition}

\begin{proof}
Using the properties of modulus and the inequality (\ref{e.2.2}) we have 
\begin{align*}
\left\vert \limfunc{tr}\left( \sum_{k=1}^{n}z_{k}P_{k}A_{k}^{\ast
}B_{k}\right) \right\vert & =\left\vert \sum_{k=1}^{n}z_{k}\limfunc{tr}%
\left( P_{k}A_{k}^{\ast }B_{k}\right) \right\vert \\
& \leq \sum_{k=1}^{n}\left\vert z_{k}\right\vert \left\vert \limfunc{tr}%
\left( P_{k}A_{k}^{\ast }B_{k}\right) \right\vert \\
& \leq \sum_{k=1}^{n}\left\vert z_{k}\right\vert \left[ \limfunc{tr}\left(
P_{k}\left\vert A_{k}\right\vert ^{2}\right) \right] ^{1/2}\left[ \limfunc{tr%
}\left( P_{k}\left\vert B_{k}\right\vert ^{2}\right) \right] ^{1/2}.
\end{align*}%
Utilizing the weighted discrete Cauchy-Bunyakovsky-Schwarz inequality we
also have 
\begin{align*}
& \sum_{k=1}^{n}\left\vert z_{k}\right\vert \left[ \limfunc{tr}\left(
P_{k}\left\vert A_{k}\right\vert ^{2}\right) \right] ^{1/2}\left[ \limfunc{tr%
}\left( P_{k}\left\vert B_{k}\right\vert ^{2}\right) \right] ^{1/2} \\
& \leq \left( \sum_{k=1}^{n}\left\vert z_{k}\right\vert \left( \left[ 
\limfunc{tr}\left( P_{k}\left\vert A_{k}\right\vert ^{2}\right) \right]
^{1/2}\right) ^{2}\right) ^{1/2}\left( \sum_{k=1}^{n}\left\vert
z_{k}\right\vert \left( \left[ \limfunc{tr}\left( P_{k}\left\vert
B_{k}\right\vert ^{2}\right) \right] ^{1/2}\right) ^{2}\right) ^{1/2} \\
& =\left( \sum_{k=1}^{n}\left\vert z_{k}\right\vert \limfunc{tr}\left(
P_{k}\left\vert A_{k}\right\vert ^{2}\right) \right) ^{1/2}\left(
\sum_{k=1}^{n}\left\vert z_{k}\right\vert \limfunc{tr}\left( P_{k}\left\vert
B_{k}\right\vert ^{2}\right) \right) ^{1/2} \\
& =\left( \limfunc{tr}\left( \sum_{k=1}^{n}\left\vert z_{k}\right\vert
P_{k}\left\vert A_{k}\right\vert ^{2}\right) \right) ^{1/2}\left( \limfunc{tr%
}\left( \sum_{k=1}^{n}\left\vert z_{k}\right\vert P_{k}\left\vert
B_{k}\right\vert ^{2}\right) \right) ^{1/2},
\end{align*}%
which is equivalent to the desired result (\ref{e.3.1}).
\end{proof}

We consider the functional for $n$-tuples of nonnegative operators as
follows:%
\begin{align}
\sigma _{\mathbf{A},\mathbf{B}}\left( \mathbf{P}\right) & :=\left[ \limfunc{%
tr}\left( \sum_{k=1}^{n}P_{k}\left\vert A_{k}\right\vert ^{2}\right) \right]
^{1/2}\left[ \limfunc{tr}\left( \sum_{k=1}^{n}P_{k}\left\vert
B_{k}\right\vert ^{2}\right) \right] ^{1/2}  \label{e.3.2} \\
& -\left\vert \limfunc{tr}\left( \sum_{k=1}^{n}P_{k}A_{k}^{\ast
}B_{k}\right) \right\vert .  \notag
\end{align}%
Utilising a similar argument to the one in Theorem \ref{t.2.2} we can state:

\begin{proposition}
\label{p.3.2}Let $\mathbf{A}=\left( A_{1},...,A_{n}\right) ,$ $\mathbf{B}%
=\left( B_{1},...,B_{n}\right) \in \mathcal{B}_{2}^{\left( n\right) }\left(
H\right) .$

(i) For any $\mathbf{P},$ $\mathbf{Q}\in \mathcal{B}_{+}^{\left( n\right)
}\left( H\right) $ we have%
\begin{equation}
\sigma _{\mathbf{A},\mathbf{B}}\left( \mathbf{P}+\mathbf{Q}\right) \geq
\sigma _{\mathbf{A},\mathbf{B}}\left( \mathbf{P}\right) +\sigma _{\mathbf{A},%
\mathbf{B}}\left( \mathbf{Q}\right) \left( \geq 0\right) ,  \label{e.3.3}
\end{equation}%
namely, $\sigma _{\mathbf{A},\mathbf{B}}$ is a superadditive functional on $%
\mathcal{B}_{+}^{\left( n\right) }\left( H\right) ;$

(ii) For any $\mathbf{P},$ $\mathbf{Q}\in \mathcal{B}_{+}^{\left( n\right)
}\left( H\right) $ with $\mathbf{P}\geq \mathbf{Q,}$ namely $P_{k}\geq Q_{k}$
for all $k\in \left\{ 1,...,n\right\} $ we have%
\begin{equation}
\sigma _{\mathbf{A},\mathbf{B}}\left( \mathbf{P}\right) \geq \sigma _{%
\mathbf{A},\mathbf{B}}\left( \mathbf{Q}\right) \left( \geq 0\right) ,
\label{e.3.4}
\end{equation}%
namely, $\sigma _{\mathbf{A},\mathbf{B}}$ is a monotonic nondecreasing
functional on $\mathcal{B}_{+}^{\left( n\right) }\left( H\right) ;$

(iii) If $\mathbf{P},\mathbf{Q}\in \mathcal{B}_{+}^{\left( n\right) }\left(
H\right) $ and there exist the constants $M>m>0$ such that $M\mathbf{Q}\geq $
$\mathbf{P}\geq m\mathbf{Q}$ then%
\begin{equation}
M\sigma _{\mathbf{A},\mathbf{B}}\left( \mathbf{Q}\right) \geq \sigma _{%
\mathbf{A},\mathbf{B}}\left( \mathbf{P}\right) \geq m\sigma _{\mathbf{A},%
\mathbf{B}}\left( \mathbf{Q}\right) \left( \geq 0\right) .  \label{e.3.5}
\end{equation}
\end{proposition}

If $\mathbf{P=}\left( p_{1}1_{H},...,p_{n}1_{H}\right) $ with $p_{k}\geq 0,$ 
$k\in \left\{ 1,...,n\right\} $ then the functional of nonnegative weights $%
\mathbf{p=}\left( p_{1},...,p_{n}\right) $ defined by 
\begin{align}
\sigma _{\mathbf{A},\mathbf{B}}\left( \mathbf{p}\right) & :=\left[ \limfunc{%
tr}\left( \sum_{k=1}^{n}p_{k}\left\vert A_{k}\right\vert ^{2}\right) \right]
^{1/2}\left[ \limfunc{tr}\left( \sum_{k=1}^{n}p_{k}\left\vert
B_{k}\right\vert ^{2}\right) \right] ^{1/2}  \label{e.3.6} \\
& -\left\vert \limfunc{tr}\left( \sum_{k=1}^{n}p_{k}A_{k}^{\ast
}B_{k}\right) \right\vert .  \notag
\end{align}%
has the same properties as in (\ref{e.3.3})-(\ref{e.3.5}).

Moreover, we have the simple bounds:%
\begin{align}
& \max_{k\in \left\{ 1,...,n\right\} }\left\{ p_{k}\right\} \left\{ \left[ 
\limfunc{tr}\left( \sum_{k=1}^{n}\left\vert A_{k}\right\vert ^{2}\right) %
\right] ^{1/2}\left[ \limfunc{tr}\left( \sum_{k=1}^{n}\left\vert
B_{k}\right\vert ^{2}\right) \right] ^{1/2}\right.  \label{e.3.7} \\
& \left. -\left\vert \limfunc{tr}\left( \sum_{k=1}^{n}A_{k}^{\ast
}B_{k}\right) \right\vert \right\}  \notag \\
& \geq \left[ \limfunc{tr}\left( \sum_{k=1}^{n}p_{k}\left\vert
A_{k}\right\vert ^{2}\right) \right] ^{1/2}\left[ \limfunc{tr}\left(
\sum_{k=1}^{n}p_{k}\left\vert B_{k}\right\vert ^{2}\right) \right]
^{1/2}-\left\vert \limfunc{tr}\left( \sum_{k=1}^{n}p_{k}A_{k}^{\ast
}B_{k}\right) \right\vert  \notag \\
& \geq \min_{k\in \left\{ 1,...,n\right\} }\left\{ p_{k}\right\} \left\{ 
\left[ \limfunc{tr}\left( \sum_{k=1}^{n}\left\vert A_{k}\right\vert
^{2}\right) \right] ^{1/2}\left[ \limfunc{tr}\left( \sum_{k=1}^{n}\left\vert
B_{k}\right\vert ^{2}\right) \right] ^{1/2}\right.  \notag \\
& \left. -\left\vert \limfunc{tr}\left( \sum_{k=1}^{n}A_{k}^{\ast
}B_{k}\right) \right\vert \right\} .  \notag
\end{align}

\section{Inequalities for Power Series of Operators}

Denote by:%
\begin{equation*}
D(0,R)=\left\{ 
\begin{array}{ll}
\{z\in \mathbb{C}:\left\vert z\right\vert <R\}, & \quad \text{if $R<\infty $}
\\ 
\mathbb{C}, & \quad \text{if $R=\infty $},%
\end{array}%
\right.
\end{equation*}%
and consider the functions:%
\begin{equation*}
\lambda \mapsto f(\lambda ):D(0,R)\rightarrow \mathbb{C},\text{ }f(\lambda
):=\sum_{n=0}^{\infty }\alpha _{n}\lambda ^{n}
\end{equation*}%
and 
\begin{equation*}
\lambda \mapsto f_{a}(\lambda ):D(0,R)\rightarrow \mathbb{C},\text{ }%
f_{a}(\lambda ):=\sum_{n=0}^{\infty }\left\vert \alpha _{n}\right\vert
\lambda ^{n}.
\end{equation*}

As some natural examples that are useful for applications, we can point out
that, if 
\begin{align}
f\left( \lambda \right) & =\sum_{n=1}^{\infty }\frac{\left( -1\right) ^{n}}{n%
}\lambda ^{n}=\ln \frac{1}{1+\lambda },\text{ }\lambda \in D\left(
0,1\right) ;  \label{E1} \\
g\left( \lambda \right) & =\sum_{n=0}^{\infty }\frac{\left( -1\right) ^{n}}{%
\left( 2n\right) !}\lambda ^{2n}=\cos \lambda ,\text{ }\lambda \in \mathbb{C}%
\text{;}  \notag \\
h\left( \lambda \right) & =\sum_{n=0}^{\infty }\frac{\left( -1\right) ^{n}}{%
\left( 2n+1\right) !}\lambda ^{2n+1}=\sin \lambda ,\text{ }\lambda \in 
\mathbb{C}\text{;}  \notag \\
l\left( \lambda \right) & =\sum_{n=0}^{\infty }\left( -1\right) ^{n}\lambda
^{n}=\frac{1}{1+\lambda },\text{ }\lambda \in D\left( 0,1\right) ;  \notag
\end{align}%
then the corresponding functions constructed by the use of the absolute
values of the coefficients are%
\begin{align}
f_{a}\left( \lambda \right) & =\sum_{n=1}^{\infty }\frac{1}{n}\lambda
^{n}=\ln \frac{1}{1-\lambda },\text{ }\lambda \in D\left( 0,1\right) ;
\label{E2} \\
g_{a}\left( \lambda \right) & =\sum_{n=0}^{\infty }\frac{1}{\left( 2n\right)
!}\lambda ^{2n}=\cosh \lambda ,\text{ }\lambda \in \mathbb{C}\text{;}  \notag
\\
h_{a}\left( \lambda \right) & =\sum_{n=0}^{\infty }\frac{1}{\left(
2n+1\right) !}\lambda ^{2n+1}=\sinh \lambda ,\text{ }\lambda \in \mathbb{C}%
\text{;}  \notag \\
l_{a}\left( \lambda \right) & =\sum_{n=0}^{\infty }\lambda ^{n}=\frac{1}{%
1-\lambda },\text{ }\lambda \in D\left( 0,1\right) .  \notag
\end{align}%
Other important examples of functions as power series representations with
nonnegative coefficients are:%
\begin{align}
\exp \left( \lambda \right) & =\sum_{n=0}^{\infty }\frac{1}{n!}\lambda
^{n}\qquad \lambda \in \mathbb{C}\text{,}  \label{E3} \\
\frac{1}{2}\ln \left( \frac{1+\lambda }{1-\lambda }\right) &
=\sum_{n=1}^{\infty }\frac{1}{2n-1}\lambda ^{2n-1},\qquad \lambda \in
D\left( 0,1\right) ;  \notag \\
\sin ^{-1}\left( \lambda \right) & =\sum_{n=0}^{\infty }\frac{\Gamma \left(
n+\frac{1}{2}\right) }{\sqrt{\pi }\left( 2n+1\right) n!}\lambda
^{2n+1},\qquad \lambda \in D\left( 0,1\right) ;  \notag \\
\tanh ^{-1}\left( \lambda \right) & =\sum_{n=1}^{\infty }\frac{1}{2n-1}%
\lambda ^{2n-1},\qquad \lambda \in D\left( 0,1\right)  \notag \\
_{2}F_{1}\left( \alpha ,\beta ,\gamma ,\lambda \right) & =\sum_{n=0}^{\infty
}\frac{\Gamma \left( n+\alpha \right) \Gamma \left( n+\beta \right) \Gamma
\left( \gamma \right) }{n!\Gamma \left( \alpha \right) \Gamma \left( \beta
\right) \Gamma \left( n+\gamma \right) }\lambda ^{n},\alpha ,\beta ,\gamma
>0,  \notag \\
\lambda & \in D\left( 0,1\right) ;  \notag
\end{align}%
where $\Gamma $ is \textit{Gamma function}.

\begin{proposition}
\label{p.3.3}Let $f(\lambda ):=\sum_{n=0}^{\infty }\alpha _{n}\lambda ^{n}$
be a power series with complex coefficients and convergent on the open disk $%
D\left( 0,R\right) ,$ $R>0.$ If $\left( H,\left\langle \cdot ,\cdot
\right\rangle \right) $ is a separable infinite-dimensional Hilbert space
and $A,$ $B\in \mathcal{B}_{1}\left( H\right) $ are positive operators with $%
\limfunc{tr}\left( A\right) ,$ $\limfunc{tr}\left( B\right) <R^{1/2}$, then%
\begin{equation}
\left\vert \limfunc{tr}\left( f\left( AB\right) \right) \right\vert ^{2}\leq
f_{a}^{2}\left( \limfunc{tr}A\limfunc{tr}B\right) \leq f_{a}\left( (\limfunc{%
tr}A)^{2}\right) f_{a}\left( (\limfunc{tr}B)^{2}\right) .  \label{e.3.7.1}
\end{equation}
\end{proposition}

\begin{proof}
By the inequality (\ref{e.1.10}) for the positive operators $A,$ $B\in 
\mathcal{B}_{1}\left( H\right) $ we have%
\begin{align}
\left\vert \limfunc{tr}\left[ \sum_{k=0}^{n}\alpha _{k}(AB)^{k}\right]
\right\vert & =\left\vert \sum_{k=0}^{n}\alpha _{k}\limfunc{tr}\left[
(AB)^{k}\right] \right\vert  \label{e.3.7.2} \\
& \leq \sum_{k=0}^{n}\left\vert \alpha _{k}\right\vert \left\vert \limfunc{tr%
}\left[ (AB)^{k}\right] \right\vert =\sum_{k=0}^{n}\left\vert \alpha
_{k}\right\vert \limfunc{tr}\left[ (AB)^{k}\right]  \notag \\
& \leq \sum_{k=0}^{n}\left\vert \alpha _{k}\right\vert (\limfunc{tr}A)^{k}(%
\limfunc{tr}B)^{k}=\sum_{k=0}^{n}\left\vert \alpha _{k}\right\vert (\limfunc{%
tr}A\limfunc{tr}B)^{k}.  \notag
\end{align}%
Utilising the weighted Cauchy-Bunyakovsky-Schwarz inequality for sums we have%
\begin{equation}
\sum_{k=0}^{n}\left\vert \alpha _{k}\right\vert (\limfunc{tr}A)^{k}(\limfunc{%
tr}B)^{k}\leq \left( \sum_{k=0}^{n}\left\vert \alpha _{k}\right\vert (%
\limfunc{tr}A)^{2k}\right) ^{1/2}\left( \sum_{k=0}^{n}\left\vert \alpha
_{k}\right\vert (\limfunc{tr}B)^{2k}\right) ^{1/2}.  \label{e.3.7.3}
\end{equation}%
Then by (\ref{e.3.7.2}) and (\ref{e.3.7.3}) we have%
\begin{align}
\left\vert \limfunc{tr}\left[ \sum_{k=0}^{n}\alpha _{k}(AB)^{k}\right]
\right\vert ^{2}& \leq \left[ \sum_{k=0}^{n}\left\vert \alpha
_{k}\right\vert (\limfunc{tr}A\limfunc{tr}B)^{k}\right] ^{2}  \label{e.3.7.4}
\\
& \leq \sum_{k=0}^{n}\left\vert \alpha _{k}\right\vert \left[ (\limfunc{tr}%
A)^{2}\right] ^{k}\sum_{k=0}^{n}\left\vert \alpha _{k}\right\vert \left[ (%
\limfunc{tr}B)^{2}\right] ^{k}  \notag
\end{align}%
for $n\geq 1.$

Since $0\leq \limfunc{tr}\left( A\right) ,\limfunc{tr}\left( B\right)
<R^{1/2},$ the numerical series 
\begin{equation*}
\sum_{k=0}^{\infty }\left\vert \alpha _{k}\right\vert (\limfunc{tr}A\limfunc{%
tr}B)^{k},\text{ }\sum_{k=0}^{\infty }\left\vert \alpha _{k}\right\vert %
\left[ (\limfunc{tr}A)^{2}\right] ^{k}\text{ and }\sum_{k=0}^{\infty
}\left\vert \alpha _{k}\right\vert \left[ (\limfunc{tr}B)^{2}\right] ^{k}
\end{equation*}%
are convergent.

Also, since $0\leq \limfunc{tr}(AB)\leq \limfunc{tr}\left( A\right) \limfunc{%
tr}\left( B\right) <R,$ the operator series $\sum_{k=0}^{\infty }\alpha
_{k}(AB)^{k}$ is convergent in $\mathcal{B}_{1}\left( H\right) .$

Letting $n\rightarrow \infty $ in (\ref{e.3.7.4}) and utilizing the
continuity property of $\limfunc{tr}\left( \cdot \right) $ on $\mathcal{B}%
_{1}\left( H\right) $ we get the desired result (\ref{e.3.7.1}).
\end{proof}

\begin{example}
\label{Ex.3.0} a) If we take in (\ref{e.3.7.1}) $f(\lambda )=\left( 1\pm
\lambda \right) ^{-1}$, $\left\vert \lambda \right\vert <1$ then we get the
inequality%
\begin{equation}
\left\vert \limfunc{tr}\left( \left( 1_{H}\pm AB\right) ^{-1}\right)
\right\vert ^{2}\leq \left( 1-\left( \limfunc{tr}A\right) ^{2}\right)
^{-1}\left( 1-\left( \limfunc{tr}B\right) ^{2}\right) ^{-1}  \label{e.3.7.5}
\end{equation}%
for any $A,$ $B\in \mathcal{B}_{1}\left( H\right) $ positive operators with $%
\limfunc{tr}\left( A\right) ,$ $\limfunc{tr}\left( B\right) <1.$

b) If we take in (\ref{e.3.7.1}) $f(\lambda )=\ln \left( 1\pm \lambda
\right) ^{-1}$, $\left\vert \lambda \right\vert <1$, then we get the
inequality%
\begin{equation}
\left\vert \limfunc{tr}\left( \ln \left( 1_{H}\pm AB\right) ^{-1}\right)
\right\vert ^{2}\leq \ln \left( 1-\left( \limfunc{tr}A\right) ^{2}\right)
^{-1}\ln \left( 1-\left( \limfunc{tr}B\right) ^{2}\right) ^{-1}
\label{e.3.7.6}
\end{equation}%
for any $A,$ $B\in \mathcal{B}_{1}\left( H\right) $ positive operators with $%
\limfunc{tr}\left( A\right) ,$ $\limfunc{tr}\left( B\right) <1.$
\end{example}

We have the following result as well:

\begin{theorem}
\label{t.2.4}Let $f(\lambda ):=\sum_{n=0}^{\infty }\alpha _{n}\lambda ^{n}$
be a power series with complex coefficients and convergent on the open disk $%
D\left( 0,R\right) ,$ $R>0.$ If $A,$ $B\in \mathcal{B}_{2}\left( H\right) $
are normal operators with $A^{\ast }B=BA^{\ast }$ and $\limfunc{tr}\left(
\left\vert A\right\vert ^{2}\right) ,$ $\limfunc{tr}\left( \left\vert
B\right\vert ^{2}\right) <R$ then we have the inequality%
\begin{equation}
\left\vert \limfunc{tr}\left( f\left( A^{\ast }B\right) \right) \right\vert
^{2}\leq \limfunc{tr}\left( f_{a}\left( \left\vert A\right\vert ^{2}\right)
\right) \limfunc{tr}\left( f_{a}\left( \left\vert B\right\vert ^{2}\right)
\right) .  \label{e.3.8}
\end{equation}
\end{theorem}

\begin{proof}
From the inequality (\ref{e.3.1}) we have%
\begin{equation}
\left\vert \limfunc{tr}\left( \sum_{k=0}^{n}\alpha _{k}\left( A^{\ast
}\right) ^{k}B^{k}\right) \right\vert ^{2}\leq \limfunc{tr}\left(
\sum_{k=0}^{n}\left\vert \alpha _{k}\right\vert \left\vert A^{k}\right\vert
^{2}\right) \limfunc{tr}\left( \sum_{k=0}^{n}\left\vert \alpha
_{k}\right\vert \left\vert B^{k}\right\vert ^{2}\right) .  \label{e.3.9}
\end{equation}%
Since $A,$ $B$ are normal operators, then we have $\left\vert
A^{k}\right\vert ^{2}=\left\vert A\right\vert ^{2k}$ and $\left\vert
B^{k}\right\vert ^{2}=\left\vert B\right\vert ^{2k}$ for any $k\geq 0.$
Also, since $A^{\ast }B=BA^{\ast }$ then we also have $\left( A^{\ast
}\right) ^{k}B^{k}=\left( A^{\ast }B\right) ^{k}$ for any $k\geq 0.$

Due to the fact that $A,$ $B\in \mathcal{B}_{2}\left( H\right) $ and $%
\limfunc{tr}\left( \left\vert A\right\vert ^{2}\right) ,$ $\limfunc{tr}%
\left( \left\vert B\right\vert ^{2}\right) <R,$ it follows that $\limfunc{tr}%
\left( A^{\ast }B\right) \leq R$ and the operator series 
\begin{equation*}
\sum_{k=0}^{\infty }\alpha _{k}\left( A^{\ast }B\right) ^{k},\text{ }%
\sum_{k=0}^{\infty }\left\vert \alpha _{k}\right\vert \left\vert
A\right\vert ^{2k}\text{ and }\sum_{k=0}^{\infty }\left\vert \alpha
_{k}\right\vert \left\vert B\right\vert ^{2k}
\end{equation*}%
are convergent in the Banach space $\mathcal{B}_{1}\left( H\right) .$

Taking the limit over $n\rightarrow \infty $ in (\ref{e.3.9}) and using the
continuity of the $\limfunc{tr}\left( \cdot \right) $ on $\mathcal{B}%
_{1}\left( H\right) $ we deduce the desired result (\ref{e.3.8}).
\end{proof}

\begin{example}
\label{Ex.3.1} a) If we take in (\ref{e.3.8}) $f(\lambda )=\left( 1\pm
\lambda \right) ^{-1}$, $\left\vert \lambda \right\vert <1$ then we get the
inequality%
\begin{equation}
\left\vert \limfunc{tr}\left( \left( 1_{H}\pm A^{\ast }B\right) ^{-1}\right)
\right\vert ^{2}\leq \limfunc{tr}\left( \left( 1-\left\vert A\right\vert
^{2}\right) ^{-1}\right) \limfunc{tr}\left( \left( 1-\left\vert B\right\vert
^{2}\right) ^{-1}\right)  \label{e.3.10}
\end{equation}%
for any $A,$ $B\in \mathcal{B}_{2}\left( H\right) $ normal operators with $%
A^{\ast }B=BA^{\ast }$ and $\limfunc{tr}\left( \left\vert A\right\vert
^{2}\right) ,$ $\limfunc{tr}\left( \left\vert B\right\vert ^{2}\right) <1.$

b) If we take in (\ref{e.3.8}) $f(\lambda )=\exp \left( \lambda \right) $, $%
\lambda \in \mathbb{C}$ then we get the inequality%
\begin{equation}
\left\vert \limfunc{tr}\left( \exp \left( A^{\ast }B\right) \right)
\right\vert ^{2}\leq \limfunc{tr}\left( \exp \left( \left\vert A\right\vert
^{2}\right) \right) \limfunc{tr}\left( \exp \left( \left\vert B\right\vert
^{2}\right) \right)  \label{e.3.11}
\end{equation}%
for any $A,$ $B\in \mathcal{B}_{2}\left( H\right) $ normal operators with $%
A^{\ast }B=BA^{\ast }.$
\end{example}

\begin{theorem}
\label{t.3.2}Let $f\left( z\right) :=\sum_{j=0}^{\infty }p_{j}z^{j}$ and $%
g\left( z\right) :=\sum_{j=0}^{\infty }q_{j}z^{j}$ be two power series with
nonnegative coefficients and convergent on the open disk $D\left( 0,R\right)
,$ $R>0.$ If $T$ and $V$ are two normal and commuting operators from $%
\mathcal{B}_{2}\left( H\right) $ with $\limfunc{tr}\left( \left\vert
T\right\vert ^{2}\right) ,$ $\limfunc{tr}\left( \left\vert V\right\vert
^{2}\right) <R,$ then%
\begin{align}
& \left[ \limfunc{tr}\left( f\left( \left\vert T\right\vert ^{2}\right)
+g\left( \left\vert T\right\vert ^{2}\right) \right) \right] ^{1/2}\left[ 
\limfunc{tr}\left( f\left( \left\vert V\right\vert ^{2}\right) +g\left(
\left\vert V\right\vert ^{2}\right) \right) \right] ^{1/2}  \label{e.3.12} \\
& -\left\vert \limfunc{tr}\left( f\left( T^{\ast }V\right) +g\left( T^{\ast
}V\right) \right) \right\vert  \notag \\
& \geq \left[ \limfunc{tr}\left( f\left( \left\vert T\right\vert ^{2}\right)
\right) \right] ^{1/2}\left[ \limfunc{tr}\left( f\left( \left\vert
V\right\vert ^{2}\right) \right) \right] ^{1/2}-\left\vert \limfunc{tr}%
\left( f\left( T^{\ast }V\right) \right) \right\vert  \notag \\
& +\left[ \limfunc{tr}\left( g\left( \left\vert T\right\vert ^{2}\right)
\right) \right] ^{1/2}\left[ \limfunc{tr}\left( g\left( \left\vert
V\right\vert ^{2}\right) \right) \right] ^{1/2}-\left\vert \limfunc{tr}%
\left( g\left( T^{\ast }V\right) \right) \right\vert \left( \geq 0\right) . 
\notag
\end{align}

Moreover, if $p_{j}\geq q_{j}$ for any $j\in \mathbb{N}$, then, with the
above assumptions on $T$ and $V,$ we have 
\begin{align}
& \left[ \limfunc{tr}\left( f\left( \left\vert T\right\vert ^{2}\right)
\right) \right] ^{1/2}\left[ \limfunc{tr}\left( f\left( \left\vert
V\right\vert ^{2}\right) \right) \right] ^{1/2}-\left\vert \limfunc{tr}%
\left( f\left( T^{\ast }V\right) \right) \right\vert  \label{e.3.13} \\
& \geq \left[ \limfunc{tr}\left( g\left( \left\vert T\right\vert ^{2}\right)
\right) \right] ^{1/2}\left[ \limfunc{tr}\left( g\left( \left\vert
V\right\vert ^{2}\right) \right) \right] ^{1/2}-\left\vert \limfunc{tr}%
\left( g\left( T^{\ast }V\right) \right) \right\vert \left( \geq 0\right) . 
\notag
\end{align}
\end{theorem}

\begin{proof}
Utilising the superadditivity property of the functional $\sigma _{\mathbf{A}%
,\mathbf{B}}\left( \mathbf{\cdot }\right) $ above as a function of weights $%
\mathbf{p}$ and the fact that $T$ and $V$ are two normal and commuting
operators we can state that%
\begin{align}
& \left[ \limfunc{tr}\left( \sum_{k=0}^{n}\left( p_{k}+q_{k}\right)
\left\vert T\right\vert ^{2k}\right) \right] ^{1/2}\left[ \limfunc{tr}\left(
\sum_{k=0}^{n}\left( p_{k}+q_{k}\right) \left\vert V\right\vert ^{2k}\right) %
\right] ^{1/2}  \label{e.3.14} \\
& -\left\vert \limfunc{tr}\left( \sum_{k=0}^{n}\left( p_{k}+q_{k}\right)
\left( T^{\ast }V\right) ^{k}\right) \right\vert  \notag \\
& \geq \left[ \limfunc{tr}\left( \sum_{k=0}^{n}p_{k}\left\vert T\right\vert
^{2k}\right) \right] ^{1/2}\left[ \limfunc{tr}\left(
\sum_{k=0}^{n}p_{k}\left\vert V\right\vert ^{2k}\right) \right]
^{1/2}-\left\vert \limfunc{tr}\left( \sum_{k=0}^{n}p_{k}\left( T^{\ast
}V\right) ^{k}\right) \right\vert  \notag \\
& +\left[ \limfunc{tr}\left( \sum_{k=0}^{n}q_{k}\left\vert T\right\vert
^{2k}\right) \right] ^{1/2}\left[ \limfunc{tr}\left(
\sum_{k=0}^{n}q_{k}\left\vert V\right\vert ^{2k}\right) \right]
^{1/2}-\left\vert \limfunc{tr}\left( \sum_{k=0}^{n}q_{k}\left( T^{\ast
}V\right) ^{k}\right) \right\vert  \notag
\end{align}%
for any $n\geq 1.$

Since all the series whose partial sums are involved in (\ref{e.3.14}) are
convergent in $\mathcal{B}_{1}\left( H\right) ,$ by letting $n\rightarrow
\infty $ in (\ref{e.3.14}) we get (\ref{e.3.12}).

The inequality (\ref{e.3.13}) follows by the monotonicity property of $%
\sigma _{\mathbf{A},\mathbf{B}}\left( \mathbf{\cdot }\right) $ and the
details are omitted.
\end{proof}

\begin{example}
\label{Ex.3.2}Now, observe that if we take%
\begin{equation*}
f\left( \lambda \right) =\sinh \lambda =\sum_{n=0}^{\infty }\frac{1}{\left(
2n+1\right) !}\lambda ^{2n+1}
\end{equation*}%
and 
\begin{equation*}
g\left( \lambda \right) =\cosh \lambda =\sum_{n=0}^{\infty }\frac{1}{\left(
2n\right) !}\lambda ^{2n}
\end{equation*}%
then 
\begin{equation*}
f\left( \lambda \right) +g\left( \lambda \right) =\exp \lambda
=\sum_{n=0}^{\infty }\frac{1}{n!}\lambda ^{n}
\end{equation*}
for any $\lambda \in \mathbb{C}$.

If $T$ and $V$ are two normal and commuting operators from $\mathcal{B}%
_{2}\left( H\right) $, then by (\ref{e.2.12}) we have%
\begin{align}
& \left[ \limfunc{tr}\left( \exp \left( \left\vert T\right\vert ^{2}\right)
\right) \right] ^{1/2}\left[ \limfunc{tr}\left( \exp \left( \left\vert
V\right\vert ^{2}\right) \right) \right] ^{1/2}-\left\vert \limfunc{tr}%
\left( \exp \left( T^{\ast }V\right) \right) \right\vert  \label{e.3.15} \\
& \geq \left[ \limfunc{tr}\left( \sinh \left( \left\vert T\right\vert
^{2}\right) \right) \right] ^{1/2}\left[ \limfunc{tr}\left( \sinh \left(
\left\vert V\right\vert ^{2}\right) \right) \right] ^{1/2}-\left\vert 
\limfunc{tr}\left( \sinh \left( T^{\ast }V\right) \right) \right\vert  \notag
\\
& +\left[ \limfunc{tr}\left( \cosh \left( \left\vert T\right\vert
^{2}\right) \right) \right] ^{1/2}\left[ \limfunc{tr}\left( \cosh \left(
\left\vert V\right\vert ^{2}\right) \right) \right] ^{1/2}-\left\vert 
\limfunc{tr}\left( \cosh \left( T^{\ast }V\right) \right) \right\vert \left(
\geq 0\right) .  \notag
\end{align}

Now, consider the series $\frac{1}{1-\lambda }=\sum_{n=0}^{\infty }\lambda
^{n},$ $\lambda \in D\left( 0,1\right) $ and $\ln \frac{1}{1-\lambda }%
=\sum_{n=1}^{\infty }\frac{1}{n}\lambda ^{n},$ $\lambda \in D\left(
0,1\right) $ and define $p_{n}=1,$ $n\geq 0,$ $q_{0}=0,$ $q_{n}=\frac{1}{n},$
$n\geq 1,$ then we observe that for any $n\geq 0$ we have $p_{n}\geq q_{n}.$

If $T$ and $V$ are two normal and commuting operators from $\mathcal{B}%
_{2}\left( H\right) $ with $\limfunc{tr}\left( \left\vert T\right\vert
^{2}\right) ,$ $\limfunc{tr}\left( \left\vert V\right\vert ^{2}\right) <1,$
then by (\ref{e.2.13}) we have%
\begin{align}
& \left[ \limfunc{tr}\left( \left( 1_{H}-\left\vert T\right\vert ^{2}\right)
^{-1}\right) \right] ^{1/2}\left[ \limfunc{tr}\left( \left( 1_{H}-\left\vert
V\right\vert ^{2}\right) ^{-1}\right) \right] ^{1/2}-\left\vert \limfunc{tr}%
\left( \left( 1_{H}-T^{\ast }V\right) ^{-1}\right) \right\vert
\label{e.3.16} \\
& \geq \left[ \limfunc{tr}\left( \ln \left( 1_{H}-\left\vert T\right\vert
^{2}\right) ^{-1}\right) \right] ^{1/2}\left[ \limfunc{tr}\left( \ln \left(
1_{H}-\left\vert V\right\vert ^{2}\right) ^{-1}\right) \right] ^{1/2}  \notag
\\
& -\left\vert \limfunc{tr}\left( \ln \left( 1_{H}-T^{\ast }V\right)
^{-1}\right) \right\vert \left( \geq 0\right) .  \notag
\end{align}
\end{example}

\section{Inequalities for Matrices}

We have the following result for matrices.

\begin{proposition}
\label{p.4.1}Let $f(\lambda ):=\sum_{n=0}^{\infty }\alpha _{n}\lambda ^{n}$
be a power series with complex coefficients and convergent on the open disk $%
D\left( 0,R\right) ,$ $R>0.$ If $A$ and $B$ are positive semidefinite
matrices in $M_{n}\left( \mathbb{C}\right) $ with $\limfunc{tr}\left(
A^{2}\right) ,$ $\limfunc{tr}\left( B^{2}\right) <R,$ then we have the
inequality%
\begin{equation}
\left\vert \limfunc{tr}\left[ f(AB)\right] \right\vert ^{2}\leq \limfunc{tr}%
\left[ f_{a}\left( A^{2}\right) \right] \limfunc{tr}\left[ f_{a}\left(
B^{2}\right) \right] .  \label{e.4.1}
\end{equation}%
If $\limfunc{tr}\left( A\right) ,$ $\limfunc{tr}\left( B\right) <\sqrt{R},$
then also 
\begin{equation}
\left\vert \limfunc{tr}\left[ f(AB)\right] \right\vert \leq \min \left\{ 
\limfunc{tr}\left( f_{a}\left( \left\Vert A\right\Vert B\right) \right) ,%
\limfunc{tr}\left( f_{a}\left( \left\Vert B\right\Vert A\right) \right)
\right\} .  \label{e.4.2}
\end{equation}
\end{proposition}

\begin{proof}
We observe that (\ref{e.1.11}) holds for $m=0$ with equality.

By utilizing the generalized triangle inequality for the modulus and the
inequality (\ref{e.1.11}) we have%
\begin{align}
& \left\vert \limfunc{tr}\left[ \sum_{n=0}^{m}\alpha _{n}(AB)^{n}\right]
\right\vert  \label{e.4.3} \\
& =\left\vert \sum_{n=0}^{m}\alpha _{n}\limfunc{tr}\left[ (AB)^{n}\right]
\right\vert \leq \sum_{n=0}^{m}\left\vert \alpha _{n}\right\vert \left\vert 
\limfunc{tr}\left[ (AB)^{n}\right] \right\vert  \notag \\
& =\sum_{n=0}^{m}\left\vert \alpha _{n}\right\vert \limfunc{tr}\left[
(AB)^{n}\right] \leq \sum_{n=0}^{m}\left\vert \alpha _{n}\right\vert \left[ 
\limfunc{tr}\left( A^{2n}\right) \right] ^{1/2}\left[ \limfunc{tr}\left(
B^{2n}\right) \right] ^{1/2},  \notag
\end{align}%
for any $m\geq 1.$

Applying the weighted Cauchy-Bunyakowsky-Schwarz discrete inequality we also
have%
\begin{align}
& \sum_{n=0}^{m}\left\vert \alpha _{n}\right\vert \left[ \limfunc{tr}\left(
A^{2n}\right) \right] ^{1/2}\left[ \limfunc{tr}\left( B^{2n}\right) \right]
^{1/2}  \label{e.4.4} \\
& \leq \left( \sum_{n=0}^{m}\left\vert \alpha _{n}\right\vert \left( \left[ 
\limfunc{tr}\left( A^{2n}\right) \right] ^{1/2}\right) ^{2}\right)
^{1/2}\left( \sum_{n=0}^{m}\left\vert \alpha _{n}\right\vert \left( \left[ 
\limfunc{tr}\left( B^{2n}\right) \right] ^{1/2}\right) ^{2}\right) ^{1/2} 
\notag \\
& =\left( \sum_{n=0}^{m}\left\vert \alpha _{n}\right\vert \left[ \limfunc{tr}%
\left( A^{2n}\right) \right] \right) ^{1/2}\left( \sum_{n=0}^{m}\left\vert
\alpha _{n}\right\vert \left[ \limfunc{tr}\left( B^{2n}\right) \right]
\right) ^{1/2}  \notag \\
& =\left[ \limfunc{tr}\left( \sum_{n=0}^{m}\left\vert \alpha _{n}\right\vert
A^{2n}\right) \right] ^{1/2}\left[ \limfunc{tr}\left(
\sum_{n=0}^{m}\left\vert \alpha _{n}\right\vert B^{2n}\right) \right] ^{1/2}
\notag
\end{align}%
for any $m\geq 1.$

Therefore, by (\ref{e.4.3}) and (\ref{e.4.4}) we get%
\begin{equation}
\left\vert \limfunc{tr}\left[ \sum_{n=0}^{m}\alpha _{n}(AB)^{n}\right]
\right\vert ^{2}\leq \limfunc{tr}\left( \sum_{n=0}^{m}\left\vert \alpha
_{n}\right\vert A^{2n}\right) \limfunc{tr}\left( \sum_{n=0}^{m}\left\vert
\alpha _{n}\right\vert B^{2n}\right)  \label{e.4.5}
\end{equation}%
for any $m\geq 1.$

Since $\limfunc{tr}\left( A^{2}\right) ,\limfunc{tr}\left( B^{2}\right) <R,$
then $\limfunc{tr}\left( AB\right) \leq \sqrt{\limfunc{tr}\left(
A^{2}\right) \limfunc{tr}\left( B^{2}\right) }<R$ and the series 
\begin{equation*}
\sum_{n=0}^{\infty }\alpha _{n}(AB)^{n},\text{ }\sum_{n=0}^{\infty
}\left\vert \alpha _{n}\right\vert A^{2n}\text{ and }\sum_{n=0}^{\infty
}\left\vert \alpha _{n}\right\vert B^{2n}
\end{equation*}%
are convergent in $M_{n}\left( \mathbb{C}\right) .$

Taking the limit over $m\rightarrow \infty $ in (\ref{e.4.5}) and utilizing
the continuity property of $\limfunc{tr}\left( \cdot \right) $ on $%
M_{n}\left( \mathbb{C}\right) $ we get (\ref{e.4.1}).

The inequality (\ref{e.4.2}) follows from (\ref{e.1.12}) in a similar way
and the details are omitted.
\end{proof}

\begin{example}
\label{Ex 4.1}a) If we take $f(\lambda )=\left( 1\pm \lambda \right) ^{-1}$, 
$\left\vert \lambda \right\vert <1$ then we get the inequality%
\begin{equation}
\left\vert \limfunc{tr}\left[ (I_{n}\pm AB)^{-1}\right] \right\vert ^{2}\leq 
\limfunc{tr}\left[ \left( I_{n}-A^{2}\right) ^{-1}\right] \limfunc{tr}\left[
\left( I_{n}-B^{2}\right) ^{-1}\right]  \label{e.4.6}
\end{equation}%
for any $A$ and $B$ positive semidefinite matrices in $M_{n}\left( \mathbb{C}%
\right) $ with $\limfunc{tr}\left( A^{2}\right) ,$ $\limfunc{tr}\left(
B^{2}\right) <1.$ Here $I_{n}$ is the identity matrix in $M_{n}\left( 
\mathbb{C}\right) .$

We also have the inequality%
\begin{equation}
\left\vert \limfunc{tr}\left[ (I_{n}\pm AB)^{-1}\right] \right\vert \leq
\min \left\{ \limfunc{tr}\left( \left( I_{n}-\left\Vert A\right\Vert
B\right) ^{-1}\right) ,\limfunc{tr}\left( \left( I_{n}-\left\Vert
B\right\Vert A\right) ^{-1}\right) \right\}  \label{e.4.7}
\end{equation}%
for any $A$ and $B$ positive semidefinite matrices in $M_{n}\left( \mathbb{C}%
\right) $ with $\limfunc{tr}\left( A\right) ,$ $\limfunc{tr}\left( B\right)
<1.$

b) If we take $f(\lambda )=\exp \lambda ,$ then we have the inequalities%
\begin{equation}
\left( \limfunc{tr}\left[ \exp (AB)\right] \right) ^{2}\leq \limfunc{tr}%
\left[ \exp \left( A^{2}\right) \right] \limfunc{tr}\left[ \exp \left(
B^{2}\right) \right]  \label{e.4.8}
\end{equation}%
and%
\begin{equation}
\limfunc{tr}\left[ \exp (AB)\right] \leq \min \left\{ \limfunc{tr}\left(
\exp \left( \left\Vert A\right\Vert B\right) \right) ,\limfunc{tr}\left(
\exp \left( \left\Vert B\right\Vert A\right) \right) \right\}  \label{e.4.9}
\end{equation}%
for any $A$ and $B$ positive semidefinite matrices in $M_{n}\left( \mathbb{C}%
\right) .$
\end{example}

\begin{proposition}
\label{p.4.2}Let $f(\lambda ):=\sum_{n=0}^{\infty }\alpha _{n}\lambda ^{n}$
be a power series with complex coefficients and convergent on the open disk $%
D\left( 0,R\right) ,$ $R>0.$ If $A$ and $B$ are matrices in $M_{n}\left( 
\mathbb{C}\right) $ with $\limfunc{tr}\left( \left\vert A\right\vert
^{p}\right) ,$\ $\limfunc{tr}\left( \left\vert B\right\vert ^{q}\right) <R,$
where\ $p,$ $q>1$ with $\frac{1}{p}+\frac{1}{q}=1,$ then%
\begin{eqnarray}
\left\vert \limfunc{tr}\left( f\left( \left\vert AB^{\ast }\right\vert
\right) \right) \right\vert &\leq &\limfunc{tr}\left[ f_{a}\left( \frac{%
\left\vert A\right\vert ^{p}}{p}+\frac{\left\vert B\right\vert ^{q}}{q}%
\right) \right]  \label{e.4.10} \\
&\leq &\limfunc{tr}\left[ \frac{1}{p}f_{a}\left( \left\vert A\right\vert
^{p}\right) +\frac{1}{q}f_{a}\left( \left\vert B\right\vert ^{q}\right) %
\right] .  \notag
\end{eqnarray}
\end{proposition}

\begin{proof}
The inequality (\ref{e.1.13}) holds with equality for $r=0.$

By utilizing the generalized triangle inequality for the modulus and the
inequality (\ref{e.1.13}) we have%
\begin{align}
\left\vert \limfunc{tr}\left( \sum_{n=0}^{m}\alpha _{n}\left\vert AB^{\ast
}\right\vert ^{n}\right) \right\vert & =\left\vert \sum_{n=0}^{m}\alpha _{n}%
\limfunc{tr}\left( \left\vert AB^{\ast }\right\vert ^{n}\right) \right\vert
\label{e.4.11} \\
& \leq \sum_{n=0}^{m}\left\vert \alpha _{n}\right\vert \left\vert \limfunc{tr%
}\left( \left\vert AB^{\ast }\right\vert ^{n}\right) \right\vert
=\sum_{n=0}^{m}\left\vert \alpha _{n}\right\vert \limfunc{tr}\left(
\left\vert AB^{\ast }\right\vert ^{n}\right)  \notag \\
& \leq \sum_{n=0}^{m}\left\vert \alpha _{n}\right\vert \limfunc{tr}\left[
\left( \frac{\left\vert A\right\vert ^{p}}{p}+\frac{\left\vert B\right\vert
^{q}}{q}\right) ^{n}\right]  \notag \\
& =\limfunc{tr}\left[ \sum_{n=0}^{m}\left\vert \alpha _{n}\right\vert \left( 
\frac{\left\vert A\right\vert ^{p}}{p}+\frac{\left\vert B\right\vert ^{q}}{q}%
\right) ^{n}\right]  \notag
\end{align}%
for any $m\geq 1$ and $p,q>1$ with $\frac{1}{p}+\frac{1}{q}=1.$

It is know that if $f:[0,\infty )\rightarrow \mathbb{R}$ is a convex
function, then $\limfunc{tr}f\left( \cdot \right) $ is convex on the cone $%
M_{n}^{+}\left( \mathbb{C}\right) $ of positive semidefinite matrices in $%
M_{n}\left( \mathbb{C}\right) $. Therefore, for $n\geq 1$ we have%
\begin{equation}
\limfunc{tr}\left[ \left( \frac{\left\vert A\right\vert ^{p}}{p}+\frac{%
\left\vert B\right\vert ^{q}}{q}\right) ^{n}\right] \leq \frac{1}{p}\limfunc{%
tr}\left( \left\vert A\right\vert ^{pn}\right) +\frac{1}{q}\limfunc{tr}%
\left( \left\vert B\right\vert ^{qn}\right)  \label{e.4.12}
\end{equation}%
where $p,q>1$ with $\frac{1}{p}+\frac{1}{q}=1.$

The inequality reduces to equality if $n=0.$

Then we have%
\begin{align}
\sum_{n=0}^{m}\left\vert \alpha _{n}\right\vert \limfunc{tr}\left[ \left( 
\frac{\left\vert A\right\vert ^{p}}{p}+\frac{\left\vert B\right\vert ^{q}}{q}%
\right) ^{n}\right] & \leq \sum_{n=0}^{m}\left\vert \alpha _{n}\right\vert %
\left[ \frac{1}{p}\limfunc{tr}\left( \left\vert A\right\vert ^{pn}\right) +%
\frac{1}{q}\limfunc{tr}\left( \left\vert B\right\vert ^{qn}\right) \right]
\label{e.4.13} \\
& =\limfunc{tr}\left[ \frac{1}{p}\sum_{n=0}^{m}\left\vert \alpha
_{n}\right\vert \left\vert A\right\vert ^{pn}+\frac{1}{q}\sum_{n=0}^{m}\left%
\vert \alpha _{n}\right\vert \left\vert B\right\vert ^{qn}\right]  \notag
\end{align}%
for any $m\geq 1$ and $p,q>1$ with $\frac{1}{p}+\frac{1}{q}=1.$

From (\ref{e.4.11}) and (\ref{e.4.13}) we get%
\begin{align}
\left\vert \limfunc{tr}\left( \sum_{n=0}^{m}\alpha _{n}\left\vert AB^{\ast
}\right\vert ^{r}\right) \right\vert & \leq \limfunc{tr}\left[
\sum_{n=0}^{m}\left\vert \alpha _{n}\right\vert \left( \frac{\left\vert
A\right\vert ^{p}}{p}+\frac{\left\vert B\right\vert ^{q}}{q}\right) ^{n}%
\right]  \label{e.4.14} \\
& \leq \limfunc{tr}\left[ \frac{1}{p}\sum_{n=0}^{m}\left\vert \alpha
_{n}\right\vert \left\vert A\right\vert ^{pn}+\frac{1}{q}\sum_{n=0}^{m}\left%
\vert \alpha _{n}\right\vert \left\vert B\right\vert ^{qn}\right]  \notag
\end{align}%
for any $m\geq 1$ and $p,q>1$ with $\frac{1}{p}+\frac{1}{q}=1.$

Since $\limfunc{tr}\left( \left\vert A\right\vert ^{p}\right) ,$\ $\limfunc{%
tr}\left( \left\vert B\right\vert ^{q}\right) <R,$ then all the series whose
partial sums are involved in (\ref{e.4.14}) are convergent, then by letting $%
m\rightarrow \infty $ in (\ref{e.4.14}) we deduce the desired inequality (%
\ref{e.4.10}).
\end{proof}

\begin{example}
\label{Ex 4.2}a) If we take $f(\lambda )=\left( 1\pm \lambda \right) ^{-1}$, 
$\left\vert \lambda \right\vert <1$ then we get the inequalities%
\begin{eqnarray}
\left\vert \limfunc{tr}\left( \left( I_{n}\pm \left\vert AB^{\ast
}\right\vert \right) ^{-1}\right) \right\vert &\leq &\limfunc{tr}\left( %
\left[ I_{n}-\left( \frac{\left\vert A\right\vert ^{p}}{p}+\frac{\left\vert
B\right\vert ^{q}}{q}\right) \right] ^{-1}\right)  \label{e.4.15} \\
&\leq &\limfunc{tr}\left[ \frac{1}{p}\left( I_{n}-\left\vert A\right\vert
^{p}\right) ^{-1}+\frac{1}{q}\left( I_{n}-\left\vert B\right\vert
^{q}\right) ^{-1}\right] ,  \notag
\end{eqnarray}%
where $A$ and $B$ are matrices in $M_{n}\left( \mathbb{C}\right) $ with $%
\limfunc{tr}\left( \left\vert A\right\vert ^{p}\right) ,$\ $\limfunc{tr}%
\left( \left\vert B\right\vert ^{q}\right) <1$ and\ $p,q>1$ with $\frac{1}{p}%
+\frac{1}{q}=1.$

b) If we take $f(\lambda )=\exp \lambda ,$ then we have the inequalities%
\begin{eqnarray}
\limfunc{tr}\left( \exp \left( \left\vert AB^{\ast }\right\vert \right)
\right) &\leq &\limfunc{tr}\left[ \exp \left( \frac{\left\vert A\right\vert
^{p}}{p}+\frac{\left\vert B\right\vert ^{q}}{q}\right) \right]
\label{e.4.16} \\
&\leq &\limfunc{tr}\left[ \frac{1}{p}\exp \left( \left\vert A\right\vert
^{p}\right) +\frac{1}{q}\exp \left( \left\vert B\right\vert ^{q}\right) %
\right] ,  \notag
\end{eqnarray}%
where $A$ and $B$ are matrices in $M_{n}\left( \mathbb{C}\right) $ and\ $%
p,q>1$ with $\frac{1}{p}+\frac{1}{q}=1.$
\end{example}

Finally, we can state the following result:

\begin{proposition}
\label{p.4.3}Let $f(\lambda ):=\sum_{n=0}^{\infty }\alpha _{n}\lambda ^{n}$
be a power series with complex coefficients and convergent on the open disk $%
D\left( 0,R\right) ,$ $R>0.$ If $A$ and $B$ are commuting positive
semidefinite matrices in $M_{n}\left( \mathbb{C}\right) $ with $\limfunc{tr}%
\left( A^{p}\right) ,$\ $\limfunc{tr}\left( B^{q}\right) <R,$ where\ $p,q>1$
with $\frac{1}{p}+\frac{1}{q}=1,$ then we have the inequality%
\begin{equation}
\left\vert \limfunc{tr}\left( f\left( AB\right) \right) \right\vert \leq %
\left[ \limfunc{tr}(f_{a}\left( A^{p}\right) )\right] ^{1/p}\left[ \limfunc{%
tr}(f_{a}\left( B^{q}\right) )\right] ^{1/q}.  \label{e.4.17}
\end{equation}
\end{proposition}

\begin{proof}
Since $A$ and $B$ are commuting positive semidefinite matrices in $%
M_{n}\left( \mathbb{C}\right) ,$ then by (\ref{e.1.14}) we have for any
natural number $n$ including $n=0$ that%
\begin{equation}
\limfunc{tr}(\left( AB\right) ^{n})=\limfunc{tr}(A^{n}B^{n})\leq \left[ 
\limfunc{tr}(A^{np})\right] ^{1/p}\left[ \limfunc{tr}(B^{nq})\right] ^{1/q},
\label{e.4.18}
\end{equation}%
where\ $p,q>1$ with $\frac{1}{p}+\frac{1}{q}=1.$

By (\ref{e.4.18}) and the weighted H\"{o}lder discrete inequality we have%
\begin{align*}
\left\vert \limfunc{tr}\left( \sum_{n=0}^{m}\alpha _{n}\left( AB\right)
^{n}\right) \right\vert & =\left\vert \sum_{n=0}^{m}\alpha _{n}\limfunc{tr}%
(A^{n}B^{n})\right\vert \leq \sum_{n=0}^{m}\left\vert \alpha _{n}\right\vert
\left\vert \limfunc{tr}(A^{n}B^{n})\right\vert \\
& \leq \sum_{n=0}^{m}\left\vert \alpha _{n}\right\vert \left[ \limfunc{tr}%
(A^{np})\right] ^{1/p}\left[ \limfunc{tr}(B^{nq})\right] ^{1/q} \\
& \leq \left( \sum_{n=0}^{m}\left\vert \alpha _{n}\right\vert \left( \left[ 
\limfunc{tr}(A^{np})\right] ^{1/p}\right) ^{p}\right) ^{1/p} \\
& \times \left( \sum_{n=0}^{m}\left\vert \alpha _{n}\right\vert \left( \left[
\limfunc{tr}(B^{nq})\right] ^{1/q}\right) ^{q}\right) ^{1/q} \\
& =\left( \sum_{n=0}^{m}\left\vert \alpha _{n}\right\vert \limfunc{tr}%
(A^{np})\right) ^{1/p}\left( \sum_{n=0}^{m}\left\vert \alpha _{n}\right\vert 
\limfunc{tr}(B^{nq})\right) ^{1/q} \\
& =\left( \limfunc{tr}(\sum_{n=0}^{m}\left\vert \alpha _{n}\right\vert
A^{np})\right) ^{1/p}\left( \limfunc{tr}(\sum_{n=0}^{m}\left\vert \alpha
_{n}\right\vert B^{nq})\right) ^{1/q}
\end{align*}%
where\ $p,q>1$ with $\frac{1}{p}+\frac{1}{q}=1.$

The proof follows now in a similar way with the ones from above and the
details are omitted.
\end{proof}

\begin{example}
\label{Ex 4.3}a) If we take $f(\lambda )=\left( 1\pm \lambda \right) ^{-1}$, 
$\left\vert \lambda \right\vert <1$ then we get the inequality%
\begin{equation}
\left\vert \limfunc{tr}\left( \left( I_{n}\pm AB\right) ^{-1}\right)
\right\vert \leq \left[ \limfunc{tr}(\left( I_{n}-A^{p}\right) ^{-1})\right]
^{1/p}\left[ \limfunc{tr}(\left( I_{n}-B^{q}\right) ^{-1})\right] ^{1/q},
\label{e.4.19}
\end{equation}%
for any $A$ and $B$ commuting positive semidefinite matrices in $M_{n}\left( 
\mathbb{C}\right) $ with $\limfunc{tr}\left( A^{p}\right) ,$\ $\limfunc{tr}%
\left( B^{q}\right) <1,$ where\ $p,q>1$ with $\frac{1}{p}+\frac{1}{q}=1.$

b) If we take $f(\lambda )=\exp \lambda ,$ then we have the inequality%
\begin{equation}
\limfunc{tr}\left( \exp \left( AB\right) \right) \leq \left[ \limfunc{tr}%
(\exp \left( A^{p}\right) )\right] ^{1/p}\left[ \limfunc{tr}(\exp \left(
B^{q}\right) )\right] ^{1/q},  \label{e.4.20}
\end{equation}%
for any $A$ and $B$ commuting positive semidefinite matrices in $M_{n}\left( 
\mathbb{C}\right) $ and\ $p,q>1$ with $\frac{1}{p}+\frac{1}{q}=1.$
\end{example}

\end{document}